\documentclass[10pt]{article}

\usepackage[T1]{fontenc}    % Output font encoding
\usepackage[utf8]{inputenc} % Input encoding (legacy but harmless)
\usepackage{lmodern}        % Latin Modern fonts
\usepackage{microtype}      % Improved justification, kerning, and line breaking
\usepackage{bm}             % Bold math symbols

\usepackage[margin=1in]{geometry} % Set page margins and layout

\usepackage{amsmath,amsthm,amssymb,mathtools} % Math environments/symbols

\usepackage{graphicx} % Include graphics
\graphicspath{{images/}}

\usepackage[shortlabels]{enumitem} % Control list formatting

\usepackage{algorithm}      % Floating algorithm environment (captioned like figures/tables)
\usepackage{algpseudocode}  % Structured pseudocode commands built on algorithmicx

\usepackage[section]{placeins} % Provides \FloatBarrier to prevent floats entering sensitive regions 

\usepackage{xcolor}
\usepackage{tabularray}
\UseTblrLibrary{booktabs}
\usepackage{tcolorbox} % Colored/shaded boxes
\tcbuselibrary{skins,breakable} % Enhanced drawing + page breaks

\tcbset{
    coolboxgrey/.style={
        enhanced,
        sharp corners,
        breakable,                % Allows page break
        borderline west={2pt}{0pt}{gray},
        colback=gray!10,          % Background color
        colframe=gray!10,         % Frame (border) color
        boxrule=0pt,
        before upper={
            \postdisplaypenalty=10000\relax
            \setlength{\abovedisplayskip}{6pt}
            \setlength{\belowdisplayskip}{6pt}
            \setlength{\abovedisplayshortskip}{4pt}
            \setlength{\belowdisplayshortskip}{4pt}
        },
    },
    coolboxgreydef/.style={
        coolboxgrey,
        fontupper=\normalfont,
    }
}

\newtcolorbox{whitetablebox}{
    enhanced,
    sharp corners,
    colback=white,
    colframe=white,
    boxrule=0pt,
    boxsep=0pt,
    left=0pt,
    right=0pt,
    top=0pt,
    bottom=0pt,
    before skip=8pt,
    after skip=8pt,
}

\usepackage{titlesec} % Customize section headings

\titleformat{\section}[runin]
    {\normalfont\bfseries}
    {\thesection.}
    {0.5em}
    {}
    [.]

\titleformat{\subsection}[runin]
    {\normalfont\bfseries}
    {\thesubsection.}
    {0.5em}
    {}
    [.]

\usepackage[
    backend=biber,
    style=numeric,
    maxbibnames=99,        % Show all authors, never truncate to et al.
    sorting=nyt,           % Sort by name, year, title
    giveninits=true,
    sortcites=true
]{biblatex}

\renewbibmacro*{volume+number+eid}{%
    \printfield{volume}%
    \iffieldundef{number}
        {}
        {\printtext[parens]{\printfield{number}}}%
    \setunit{\addcomma\space}%
    \printfield{eid}%
}

\usepackage{hyperref} % Clickable links and cross-references
\hypersetup{
    colorlinks,
    citecolor = blue,
    linkcolor = blue,
    urlcolor  = blue
}

\newtheorem{theorem}{Theorem}[section]
\newtheorem{proposition}[theorem]{Proposition}
\newtheorem{lemma}[theorem]{Lemma}

\theoremstyle{definition}

\newtheorem{ass}{Assumption}
\newtheorem{remark}[theorem]{Remark}

\newtheorem{example}[theorem]{Example}
\tcolorboxenvironment{theorem}{coolboxgrey}
\tcolorboxenvironment{proposition}{coolboxgrey}
\tcolorboxenvironment{lemma}{coolboxgrey}
\tcolorboxenvironment{corollary}{coolboxgrey}
\tcolorboxenvironment{remark}{coolboxgrey}
\tcolorboxenvironment{problem}{coolboxgrey}
\tcolorboxenvironment{ass}{coolboxgrey}

\tcolorboxenvironment{definition}{coolboxgreydef}
\tcolorboxenvironment{notation}{coolboxgreydef}
\tcolorboxenvironment{example}{coolboxgrey}

\newcommand{\N}{\mathbb{N}}

\newcommand{\R}{\mathbb{R}}

\newcommand{\cH}{\mathcal{H}}
\newcommand{\cG}{\mathcal{G}}
\newcommand{\cE}{\mathcal{E}}
\newcommand{\infconv}{\mathbin{\square}} % Infimal convolution symbol

\DeclareRobustCommand{\rph}[1]{\mathrel{\phantom{#1}}} % For proper alignment w.r.t. relation symbols

\DeclareMathOperator{\zer}{zer}
\DeclareMathOperator{\prox}{prox}

\DeclareMathOperator{\argmin}{argmin}
\DeclareMathOperator{\dist}{dist}
\DeclareMathOperator{\graph}{gra}

\allowdisplaybreaks

\title{Regularized extragradient method for structured bilevel optimization in continuous and discrete time}
\author{
Radu Ioan Bo\c{t}\thanks{Faculty of Mathematics, University of Vienna,
Oskar-Morgenstern-Platz 1, 1090 Vienna, Austria.
\texttt{emails:}
\{radu.bot, enis.chenchene, david.alexander.hulett\}@univie.ac.at}
\and
Enis Chenchene\footnotemark[1]
\and
David A. Hulett\footnotemark[1]
}
\date{}

\begin{document}

% --------------------------------------------------------------
%                         Start here
% --------------------------------------------------------------
\maketitle
\begin{abstract}
In a real Hilbert space, we study a bilevel optimization problem that consists in minimizing an outer convex function over the zero set of a maximally monotone operator. In the smooth setting, where the outer objective is convex and Fr\'echet differentiable and the inner operator is single-valued, continuous and monotone, we associate with the problem a first-order dynamical system that can be viewed as a monotone flow applied to a dynamically regularized operator. Under a suitable Attouch--Czarnecki-type condition for the inner level, we establish last-iterate convergence rates for both the outer and inner residuals, together with weak convergence of the trajectories to optimal solutions of the bilevel problem. In the smooth+nonsmooth setting, we enrich the outer objective with a proper, convex, and lower semicontinuous function, while the inner operator is augmented by the subdifferential of a function with the same properties. We propose a regularized proximal-extragradient algorithm in which both the forward and backward steps are performed with respect to dynamically regularized operators and functions, respectively. Under an analogous Attouch--Czarnecki-type condition formulated for the smooth+nonsmooth case, we establish last-iterate convergence rates for both the outer and inner residuals, together with weak convergence of the iterates to optimal solutions of the bilevel problem.
\end{abstract}

\par\medskip


\section{Introduction}
In this work, we will be addressing the structured bilevel optimization problem 
\begin{equation}\label{eq:bilevel-inclusion-intro}
    \begin{aligned}
        &\text{minimize} & & H(x) := h(x) + \hat{h}(x),\\
        &\text{subject to } & & 0 \in V(x) + \partial \hat{f}(x) 
    \end{aligned}
\end{equation}
where $\cH$ is a real Hilbert space, $h : \cH \to \R$ is a convex and continuously Fr\'echet differentiable function, $\hat{f}, \hat{h} : \cH \to \bar{\R}$ are proper, convex and lower semicontinuous functions, and $V : \cH \to \cH$ is a continuous and monotone operator. More detailed standing assumptions on the problem will be stated later, as they differ between the continuous-time and discrete-time settings. 

Problem \eqref{eq:bilevel-inclusion-intro} belongs to the class of simple bilevel problems, in which an outer criterion selects a single decision variable from among the solutions of an inner problem. In the potential case $V=\nabla f$, the problem reduces to the composite convex bilevel program
\begin{equation*}
    \begin{aligned}
        &\text{minimize} & & h(x) + \hat{h}(x)\\
        &\text{subject to } & & x \in \argmin(f + \hat{f}), 
    \end{aligned}
\end{equation*}
whereas a general monotone operator $V$ allows the inner problem to take the form of a variational inequality or, more generally, a monotone inclusion. Thus, our model lies at the intersection of convex bilevel optimization, hierarchical variational inequalities, and equilibrium selection. Such formulations arise naturally in optimal solution selection, Nash equilibrium selection, mathematical programs with equilibrium constraints, and optimal control problems subject to variational inequality constraints; see, among others, \cite{FacchineiPang2003, DempeDinhDuttaPandit2021, KaushikYousefian2021, LuoPangRalph1996, SamadiYousefian2025}.

A classical approach to these problems is based on a slowly vanishing Tikhonov, or viscosity, regularization. Suppose for the moment that we are in the purely smooth setting, i.e., $\hat{f} \equiv 0$, $\hat{h} \equiv 0$. The idea is to replace the inner operator by the regularized operator
\[
    V + \varepsilon \nabla h
\]
and let $\varepsilon\downarrow 0$ along the trajectory or the iterates, depending on whether one considers the continuous- or discrete-time setting, respectively. This idea goes back to Tikhonov's regularization method \cite{Tikhonov1963} and was further developed in the context of bilevel optimization and monotone operator theory in works by Cabot \cite{Cabot2005}, Solodov \cite{Solodov2007-2}, Attouch and Czarnecki \cite{AttouchCzarnecki2010}, among others \cite{ AttouchCabotCzarnecki2018, AttouchCzarneckiPeypouquet2011, BotCsetnek2014}. A key requirement in the analysis is that the regularization parameter vanish slowly enough to retain the influence of the outer objective, yet fast enough to enforce asymptotic feasibility for the inner problem. The Attouch--Czarnecki summability/integrability condition, introduced in \cite{AttouchCzarnecki2010}, provides a precise formulation of this balance and has since become a central tool in the convergence analysis of diagonal methods.

For convex bilevel optimization, several recent works have developed first-order methods with convergence guarantees under increasingly weak assumptions. The nonsmooth and non-strongly convex setting has been investigated, among others, by Doron and Shtern \cite{DoronShtern2023}, Merchav and Sabach \cite{MerchavSabach2023}, Latafat, Themelis, Villa and Patrinos \cite{LatafatThemelisVillaPatrinos2025}, Shtern and Taiwo \cite{ShternTaiwo2025}, and Bo\c t, Chenchene, Csetnek and Hulett \cite{BotChencheneCsetnekHulett2025}. Projection-free and online optimization approaches have also been proposed in \cite{GiangTranHoNguyenLee2025, ShenHoNguyenKilincKarzan2023}. These works primarily exploit the potential structure of both the outer and inner problems, namely, their formulation in terms of functions rather than general operators. In contrast, the present work accommodates a monotone, possibly non-potential, operator $V$ in the inner problem, which calls for a different measure of inner optimality.

The variational inequality literature is also closely related. Extragradient and mirror-prox methods are the standard first-order tools for variational inequalities governed by Lipschitz continuous and monotone operators, beginning with Korpelevich \cite{Korpelevich1976} and Nemirovski \cite{Nemirovski2004}, with further developments such as \cite{Malitsky2020, HsiehIutzelerMalickMertikopoulos2019}. In the hierarchical setting, one seeks a solution to an outer variational inequality over the solution set of an inner one. Algorithms for such nested problems have been studied, for instance, in \cite{ThongTrietLiDong2020,VanHieuMoudafi2021,LamparielloPrioriSagratella2022,SamadiYousefian2025,AlvesChenFukuda2025,MarschnerStaudigl2025,DvurechenskyMarschnerShternStaudigl2026}. Our nonsmooth setting can be viewed as a structured special case of the framework considered by Dvurechensky, Marschner, Shtern and Staudigl \cite{DvurechenskyMarschnerShternStaudigl2026}. Whereas their outer problem is formulated as a hemi-variational inequality, ours is governed by a smooth+nonsmooth convex objective. This additional structure provides access to a natural outer residual, enabling us to establish last-iterate convergence rates, in contrast to the convergence guarantees in terms of restricted gap functions obtained in \cite{DvurechenskyMarschnerShternStaudigl2026}.

Our continuous-time analysis also relies on the Fitzpatrick function associated with the inner monotone operator. Introduced in \cite{Fitzpatrick1988}, the Fitzpatrick function has become a standard convex-analytic tool for representing monotone operators; see also \cite{BorweinDutta2016}. In the purely smooth setting, it provides a natural framework for expressing the Attouch--Czarnecki condition, see also \cite{BotCsetnek2014, BotCsetnek2016JMAA}, while in the composite setting it is replaced by an analogous construction. In the potential case, the resulting condition reduces to the more familiar Fenchel-type summability condition involving the conjugate of the inner objective gap.

\subsection{Contributions}

We study both a continuous-time dynamical system and a fully discrete proximal-extragradient method for \eqref{eq:bilevel-inclusion-intro}. In the continuous- and discrete-time settings, respectively, the regularization parameter is chosen to decay polynomially according to
\[
    \varepsilon(t)=\frac{c}{t^\delta},
    \qquad
    \varepsilon_k=\frac{c}{(k + a)^\delta},
\]
for an exponent $\delta>0$, and constants $c, a > 0$. The main contributions are the following.

$\bullet$ \textbf{A continuous-time system.} First, in the smooth case $\hat f=\hat h=0$, that is, when we minimize $h$ subject to the solutions of the monotone equation $V(x) = 0$, we associate with the problem the non-autonomous monotone flow dynamics
\[
    \dot x(t)+\frac{d}{dt}\big(V(x(t))+\varepsilon(t)\nabla h(x(t))\big)
    +V(x(t))+\varepsilon(t)\nabla h(x(t))=0.
\]
Under a mild standing assumption, we derive ergodic convergence statements for the trajectories generated by this system. Under an Attouch--Czarnecki-type integrability condition for the inner level, we establish weak convergence of the generated trajectories to solutions of the bilevel problem, together with convergence rates
\[
    \| V(x(t))\| = \mathcal{O}\left(\frac{1}{t^{\delta}}\right), \quad \langle x(t) - x^{*}, V(x(t))\rangle = o \left(\frac{1}{\sqrt{t}}\right), \quad | h(x(t)) - h(x^{*})| = o\left(\frac{1}{t^{\frac{1}{2} - \delta}}\right) 
\]
as $t\to +\infty$.

$\bullet$ \textbf{A regularized extragradient method.} In the discrete setting, we allow for nonsmooth terms at both the outer and inner levels. Namely, we consider \eqref{eq:bilevel-inclusion-intro}, where we minimize $H =h + \hat h$ subject to the set of solution of the monotone inclusion $0 \in V(x) + \partial \hat f(x)$. We propose a regularized proximal-extragradient algorithm in which the forward operator and the backward function are regularized simultaneously
\begin{equation*}
   (\forall k \geq 0) \quad \left\{
        \begin{aligned}
            y_{k} &= \prox_{s(\hat{f} + \varepsilon_{k} \hat{h})} \Bigl( x_{k} - s\Bigl( V(x_{k}) + \varepsilon_{k} \nabla h(x_{k})\Bigr)\Bigr), \\
            x_{k + 1} &= \prox_{s(\hat{f} + \varepsilon_{k} \hat{h})} \Bigl( x_{k} - s \Bigl( V(y_{k}) + \varepsilon_{k} \nabla h(y_{k})\Bigr)\Bigr).
        \end{aligned}
    \right.
\end{equation*}
Again, under an Attouch--Czarnecki-type summability condition for the inner level, we establish weak convergence of the iterates to solutions of the bilevel problem, together with convergence rates
\begin{gather*}
    \dist(0, V(x_{k}) + \partial \hat f(x_k)) = \mathcal{O}\left( \frac{1}{k^{\delta}}\right), \quad \langle x_{k} - x^{*}, V(x_{k})\rangle + \hat{f}(x_{k}) - \hat{f}(x^{*}) = o\left(\frac{1}{\sqrt{k}}\right), \\
    | H(x_{k}) - H(x^{*})| = o\left( \frac{1}{k^{\frac{1}{2} - \delta}}\right)
\end{gather*}
as $k\to +\infty$.

\subsection{A motivating bilevel optimization problem}

One particular class of bilevel optimization problems that motivates the investigations in this paper consists of problems with a structured convex minimization problem at the lower level
\begin{equation}\label{eq:inner-level-Fenchel-dual}
    \begin{aligned}
        &\text{minimize} & & H(x) := h(x) + \hat{h}(x),\\
        &\text{subject to } & & x \in \argmin_{z} \{ f(z) + g(Az) + m(z)\}, 
    \end{aligned}
\end{equation}
where $\cG$ is another real Hilbert space, $h, m : \cH \to \R$ are convex and Fr\'echet differentiable functions, $f, \hat{h} : \cH \to \bar{\R}$ and $g : \cG \to \bar{\R}$ are proper, convex and lower semicontinuous functions, and $A : \cH \to \cG$ is a linear continuous operator. Under a mild regularity condition (see, for instance, \cite{BauschkeCombettes2017, Bot2010}), a point $x^{*} \in \cH$ is an optimal solution to the inner level if and only if there exists a Lagrange multiplier $\lambda^{*}\in \cG$ such that 
\[
    \left\{
        \begin{aligned}
            0 &\in \partial f(x^{*}) + A^{*}\lambda^{*} + \nabla m(x^{*}), \\
            Ax^{*} &\in \partial g^{*}(\lambda^{*}),
        \end{aligned}
    \right.
\]
It follows that, provided such a regularity condition holds,  \eqref{eq:inner-level-Fenchel-dual} is equivalent to 
\begin{equation}\label{eq:Fenchel-bilevel-equivalent}
    \begin{aligned}
        &\text{minimize} & & \bm{H}(x, \lambda) := \bm{h}(x, \lambda) + \hat{\bm{h}}(x, \lambda),\\
        &\text{subject to } & & 0 \in \bm{V}(x, \lambda) + \partial\hat{\bm{f}}(x, \lambda)
    \end{aligned}
\end{equation}
where 
\begin{gather*}
    \bm{h}(x, \lambda) := h(x), \ \hat{\bm{h}}(x, \lambda) := \hat{h}(x), \\
    \bm{V}(x, \lambda) := \Bigl( \nabla m(x) + A^{*}\lambda, -Ax\Bigr), \ \hat{\bm{f}}(x, \lambda) := f(x) + g^{*}(\lambda), \ \ \text{hence} \ \ \partial \hat{\bm{f}}(x, \lambda) = \Bigl( \partial f(x), \partial g^{*}(\lambda)\Bigr).
\end{gather*}
In Subsection \ref{subsec:fenchel-bilevel}, we apply our regularized proximal-extragradient algorithm to \eqref{eq:Fenchel-bilevel-equivalent} and derive convergence rates for the corresponding optimality measures.

\section{A dynamical system attached to a smooth bilevel problem}
We want to address the following  bilevel optimization problem
\begin{equation}\label{eq:bilevel-where-inner-is-monotone-equation}
    \begin{aligned}
        &\text{minimize} & &h(x) \\
        &\text{subject to } & &V(x) = 0, 
    \end{aligned}
\end{equation}
where throughout this section we make the following standing assumption on $h$ and $V$. 
\begin{ass}\label{ass:standing-cont}
    $h : \cH \to \R$ is convex, continuously Fr\'echet differentiable, and bounded below, and $\nabla h$ is bounded on bounded subsets of $\cH$. The operator $V : \cH \to \cH$ is continuous and monotone.
\end{ass}
\noindent If $\nabla h$ is bounded on bounded subsets of $\cH$, then so is $h$. In particular, any Lipschitz continuous gradient $\nabla h$ is bounded on bounded subsets of $\cH$.

\subsection{The dynamical system}
We will be applying the monotone flow dynamics to the dynamically regularized operator
\[
    \Phi_{t}(x) := V(x) + \varepsilon(t)\nabla h(x),  
\]
where the regularization parameter is given by $ \varepsilon(t) := \frac{c}{t^{\delta}}$, with $c, \delta > 0$ and $t > 0$. The idea is that $\varepsilon(t) > 0$ goes to zero monotonically as $t\to \infty$, but not too fast. 

Associating with $V(x) = 0$ the monotone flow
\[
    \dot{x}(t) + V(x(t)) = 0,
\]
it is known that only the ergodic trajectory is guaranteed to converge weakly to a zero of $V$ (\cite{BaillonBrezis1976}), whereas the trajectory itself may, in general, fail to converge (\cite{CominettiPeypouquetSorin2008}).

Attouch and Svaiter (\cite{AttouchSvaiter2011}) associated with the monotone equation $V(x) = 0$ the monotone flow with a correction term
\begin{equation}\label{eq:first-order-system-just-m}
    \dot{x}(t) + \frac{d}{dt}V(x(t)) + V(x(t)) = 0.
\end{equation}
In addition to establishing existence and uniqueness of the trajectory, they proved that the residual $\|V(x(t))\|$ converges to zero and that $x(t)$ converges weakly to a zero of $V$. In fact, In fact, the residual satisfies the sharper asymptotic estimate $\|V(x(t))\| = o \left (\frac{1}{\sqrt{t}} \right)$ as $t \to + \infty$.

Continuing along this line of thought, we will investigate, for $t\geq t_{0} > 0$, the long-time behavior of the dynamics
\begin{equation}\label{eq:first-order-system-regularized-operator}
    \dot{x}(t) + \frac{d}{dt} \Phi_{t}(x(t)) + \Phi_{t}(x(t)) = 0.
\end{equation}

\subsection{An energy function and preliminary results}
In connection with \eqref{eq:first-order-system-regularized-operator}, we will study the dissipative properties and asymptotic behavior of the energy function
\[
    \cE_{\lambda}(t) := \cE_{1}(t) + \lambda\cE_{2}(t), 
\]
where
\begin{align*}
    \cE_{1}(t) &:= \frac{1}{2} \| x(t) - x^{*}\|^{2} + \langle x(t) - x^{*}, V(x(t))\rangle + \varepsilon(t) \langle x(t) - x^{*}, \nabla h(x(t))\rangle + \dot{\varepsilon}(t) (h(x(t)) - h(x^{*})), \\
    \cE_{2}(t) &:= \frac{t}{2} \| \Phi_{t}(x(t))\|^{2} + \frac{1}{2} \Bigl\| x(t) - x^{*} + \Phi_{t}(x(t))\Bigr\|^{2} + t\dot{\varepsilon}(t) (h(x(t)) - h(x^{*})),
\end{align*}
and $x^{*}$ is an optimal solution to the bilevel optimization problem \eqref{eq:bilevel-where-inner-is-monotone-equation} and $\lambda >0$. The following lemma establishes lower bounds and descent properties for $\cE_{1}$ and $\cE_{2}$.
\begin{lemma}\label{lem:full-energy-cont}
    Let $x : [t_{0}, +\infty) \to \cH$ be a solution trajectory to \eqref{eq:first-order-system-regularized-operator}, let $x^{*}$ be an optimal solution to \eqref{eq:bilevel-where-inner-is-monotone-equation}, and fix $0 < \lambda < \frac{1}{\delta}$. Then, for every $t\geq t_{0}$, the following are true:
    \begin{gather}
        \cE_{\lambda}(t) \geq \frac{1}{2} \| x(t) - x^{*}\|^{2} + \frac{\lambda t}{2} \| \Phi_{t}(x(t))\|^{2} + \zeta_{1}(t) ( h(x(t)) - h(x^{*})), \label{eq:lower-bound-E-lambda-statement}\\
        \dot{\cE}_{\lambda}(t) + \bigl\| \dot{x}(t)\bigr\|^{2} + \frac{\lambda}{2} \| \Phi_{t}(x(t))\|^{2} + \zeta_{2}(t) (h(x(t)) - h(x^{*})) \leq - (1 + \lambda)\langle x(t) - x^{*}, V(x(t))\rangle, \label{eq:descent-property-E-lambda-statement}
    \end{gather} 
    where 
    \[
        \zeta_{1}(t) := \varepsilon(t) + \dot{\varepsilon}(t) + \lambda t \dot{\varepsilon}(t), \quad \zeta_{2}(t) := \varepsilon(t) - \ddot{\varepsilon}(t) + \lambda \bigl(\varepsilon(t) - \dot{\varepsilon}(t) - t\ddot{\varepsilon}(t)\bigr), 
    \]
    and we have $\zeta_{1}(t) \asymp \varepsilon(t)$, $\zeta_{2}(t) \asymp \varepsilon(t)$ as $t\to +\infty$, where ``$\asymp$'' denotes equality of asymptotic order.
\end{lemma}
\begin{proof}
Let $t \geq t_0$.    In the first step, we focus on $\cE_{1}$. Dropping the nonnegative term corresponding to $V$ and applying the gradient inequality to $h$ yields
    \begin{align*}
        \cE_{1}(t) &\geq \frac{1}{2} \| x(t) - x^{*}\|^{2} + \varepsilon(t) \langle x(t) - x^{*}, \nabla h(x(t))\rangle + \dot{\varepsilon}(t) (h(x(t)) - h(x^{*})) \\
        &\geq  \frac{1}{2} \| x(t) - x^{*}\|^{2} + \bigl( \varepsilon(t) + \dot{\varepsilon}(t)\bigr) (h(x(t)) - h(x^{*})). 
    \end{align*}
    Combining the previous inequality with the definition of $\cE_{2}(t)$ yields the lower bound \eqref{eq:lower-bound-E-lambda-statement}. We now differentiate each summand of $\cE_{1}(t)$ separately. 
    \begin{align*}
        \frac{d}{dt} \frac{1}{2} \| x(t) - x^{*}\|^{2} &= \bigl\langle x(t) - x^{*}, \dot{x}(t)\bigr\rangle = - \bigl\langle x(t) - x^{*}, \Phi_{t}(x(t))\bigr\rangle - \left\langle x(t) - x^{*}, \frac{d}{dt} \Phi_{t}(x(t))\right\rangle \\
        &= - \langle x(t) - x^{*}, V(x(t))\rangle - \varepsilon(t) \langle x(t) - x^{*}, \nabla h(x(t))\rangle - \left\langle x(t) - x^{*}, \frac{d}{dt} V(x(t))\right\rangle \\
        &\rph{=} - \left\langle x(t) - x^{*}, \frac{d}{dt} \bigl( \varepsilon(t) \nabla h(x(t))\bigr)\right\rangle, \\
        \frac{d}{dt} \langle x(t) - x^{*}, V(x(t))\rangle &= \bigl\langle \dot{x}(t), V(x(t))\bigr\rangle + \left\langle x(t) - x^{*}, \frac{d}{dt} V(x(t))\right\rangle \\
        &= \left\langle \dot{x}(t), - \dot{x}(t) - \varepsilon(t) \nabla h(x(t)) - \dot{\varepsilon}(t) \nabla h(x(t)) - \varepsilon(t) \frac{d}{dt} \nabla h(x(t)) - \frac{d}{dt} V(x(t))\right\rangle \\
        &\rph{=} + \left\langle x(t) - x^{*}, \frac{d}{dt} V(x(t))\right\rangle \\
        &= - \bigl\| \dot{x}(t)\bigr\|^{2} - \bigl( \varepsilon(t) + \dot{\varepsilon}(t)\bigr) \bigl\langle \dot{x}(t), \nabla h(x(t))\bigr\rangle - \varepsilon(t) \left\langle \dot{x}(t), \frac{d}{dt} \nabla h(x(t))\right\rangle \\
        &\rph{=} - \left\langle \dot{x}(t), \frac{d}{dt}V(x(t))\right\rangle + \left\langle x(t) - x^{*}, \frac{d}{dt} V(x(t))\right\rangle,\\
        \frac{d}{dt} \bigl\langle x(t) - x^{*}, \varepsilon(t)\nabla h(x(t))\bigr\rangle &=  \varepsilon(t) \bigl\langle \dot{x}(t), \nabla h(x(t))\bigr\rangle + \left\langle x(t) - x^{*}, \frac{d}{dt} \bigl( \varepsilon(t) \nabla h(x(t))\bigr)\right\rangle, \\
        \frac{d}{dt} \dot{\varepsilon}(t) (h(x(t)) - h(x^{*})) &= \ddot{\varepsilon}(t) (h(x(t)) - h(x^{*})) + \dot{\varepsilon}(t) \bigl\langle \nabla h(x(t)), \dot{x}(t)\bigr\rangle.
    \end{align*}
    Putting everything together, canceling out the corresponding terms and using the gradient inequality to $h$ yields
    \begin{align}
        \dot{\cE}_{1}(t) &\leq  - \langle x(t) - x^{*}, V(x(t))\rangle - \bigl( \varepsilon(t) - \ddot{\varepsilon}(t)\bigr) (h(x(t)) - h(x^{*})) - \bigl\| \dot{x}(t)\bigr\|^{2} \nonumber\\
        &\rph{\leq}  - \varepsilon(t) \left\langle \dot{x}(t), \frac{d}{dt} \nabla h(x(t))\right\rangle - \left\langle \dot{x}(t), \frac{d}{dt} V(x(t))\right\rangle \nonumber\\
        &\leq - \langle x(t) - x^{*}, V(x(t))\rangle - \bigl( \varepsilon(t) - \ddot{\varepsilon}(t)\bigr) (h(x(t)) - h(x^{*})) - \bigl\| \dot{x}(t)\bigr\|^{2} \label{eq:full-energy-cont-1},
    \end{align}
    where we drop the nonpositive the inner product terms corresponding to $\nabla h$ and $V$ on account of their monotonicity. Regarding $\cE_{2}$, we have 
    \begin{align}
        \dot{\cE}_{2}(t) &= \frac{1}{2} \| \Phi_{t}(x(t))\|^{2} + t\left\langle \Phi_{t}(x(t)), \frac{d}{dt}\Phi_{t}(x(t))\right\rangle + \left\langle x(t) - x^{*} + \Phi_{t}(x(t)), \dot{x}(t) + \frac{d}{dt} \Phi_{t}(x(t))\right\rangle \nonumber\\
        &\rph{=} + \bigl( \dot{\varepsilon}(t) + t \ddot{\varepsilon}(t)\bigr) (h(x(t)) - h(x^{*})) + t\dot{\varepsilon}(t) \bigl\langle \nabla h(x(t)), \dot{x}(t)\bigr\rangle \nonumber\\ 
        &=  \frac{1}{2} \| \Phi_{t}(x(t))\|^{2} + t \left\langle - \dot{x}(t) - \frac{d}{dt} \Phi_{t}(x(t)), \frac{d}{dt} \Phi_{t}(x(t))\right\rangle \nonumber\\
        &\rph{=} - \Bigl\langle x(t) - x^{*} + \Phi_{t}(x(t)), \Phi_{t}(x(t))\Bigr\rangle + \bigl( \dot{\varepsilon}(t) + t \ddot{\varepsilon}(t)\bigr) (h(x(t)) - h(x^{*})) + t\dot{\varepsilon}(t) \bigl\langle \nabla h(x(t)), \dot{x}(t)\bigr\rangle \nonumber\\
        &=  - \frac{1}{2} \| \Phi_{t}(x(t))\|^{2} - t \left\langle \dot{x}(t), \frac{d}{dt} V(x(t))\right\rangle - t \left\| \frac{d}{dt}\Phi_{t}(x(t))\right\|^{2} - t \dot{\varepsilon}(t) \bigl\langle \dot{x}(t), \nabla h(x(t))\bigr\rangle \nonumber\\
        &\rph{=}  - t \varepsilon(t) \left\langle \dot{x}(t), \frac{d}{dt} \nabla h(x(t))\right\rangle - \langle x(t) - x^{*}, V(x(t))\rangle - \varepsilon(t) \langle x(t) - x^{*}, \nabla h(x(t))\rangle \nonumber\\
        &\rph{=} + \bigl( \dot{\varepsilon}(t) + t \ddot{\varepsilon}(t)\bigr) (h(x(t)) - h(x^{*})) + t \dot{\varepsilon}(t) \bigl\langle \nabla h(x(t)), \dot{x}(t)\bigr\rangle \nonumber\\
        &\leq - \frac{1}{2} \| \Phi_{t}(x(t))\|^{2} - \langle x(t) - x^{*}, V(x(t))\rangle - \bigl( \varepsilon(t) - \dot{\varepsilon}(t) - t \ddot{\varepsilon}(t)\bigr)(h(x(t)) - h(x^{*})), \label{eq:full-energy-cont-2}
    \end{align}
    where we arrive at the last inequality after canceling and dropping some nonpositive terms and applying the gradient inequality to $h$. After multiplying \eqref{eq:full-energy-cont-2} by $\lambda$ and adding it to \eqref{eq:full-energy-cont-1} we arrive at the desired lower bound \eqref{eq:descent-property-E-lambda-statement}. To finish the proof, note that, since $0 < \lambda < \frac{1}{\delta}$,
    \[
        \zeta_{1}(t) = \varepsilon(t) + \dot{\varepsilon}(t) + \lambda t \dot{\varepsilon}(t) = \left(1 - \lambda\delta - \frac{\delta}{t}\right) \varepsilon(t) \asymp \varepsilon(t)
    \]
    as $t\to +\infty$. Furthermore, we have 
    \[
        \zeta_{2}(t) = \varepsilon(t) - \ddot{\varepsilon}(t) + \lambda \bigl( \varepsilon(t) - \dot{\varepsilon}(t) - t\ddot{\varepsilon}(t)\bigr) = \left[ 1 - \frac{\delta(\delta + 1)}{t^{2}} + \lambda \left( 1 - \frac{\delta}{t} - \frac{\delta(\delta + 1)}{t}\right)\right] \varepsilon(t) \asymp \varepsilon(t)
    \]
    as $t\to +\infty$.
\end{proof}
We will now show some partial convergence guarantees which hold in the general setting of Assumption \ref{ass:standing-cont}. 
\begin{theorem}\label{thm:ergodic-rate-cont}
Let $x : [t_{0}, +\infty) \to \cH$ be a solution trajectory to \eqref{eq:first-order-system-regularized-operator} and let $x^{*}$ be an optimal solution to \eqref{eq:bilevel-where-inner-is-monotone-equation}. Define, for every $t \geq t_0$,
    \[
        x_{t}^{\textnormal{best}} := \argmin_{x \in \{ x(t), \bar{x}_{t}\}} h(x), \quad \text{where} \quad \bar{x}_{t} := \frac{1}{\int_{t_{0}}^{t} \zeta_{2}(s) ds} \int_{t_{0}}^{t} \zeta_{2}(s) x(s) ds, 
    \]
and $\zeta_{2}(t) \asymp \varepsilon(t)$ as $t \to +\infty$ was defined in Lemma \ref{lem:full-energy-cont}. Then, as $t\to +\infty$, we have the convergence rates as shown by the following table:
    \begin{whitetablebox}
        \begin{tblr}{
            width = \linewidth,
            colspec = {
                Q[c,m,wd=0.18\linewidth]
                X[0.90,c,m]
                X[1.35,c,m]
                X[0.95,c,m]
            },
            cells = {
                mode = math,
            },
            cell{1}{2-Z} = {
                cmd = \textstyle,
            },
            cell{2-Z}{2-Z} = {
                bg = gray!10,
                cmd = \displaystyle,
            },
            rows = {
                rowsep = 8pt,
            },
            columns = {
                colsep = 8pt,
            },
            cell{2-Z}{2-Z} = {
                bg = gray!10,
            },
            hline{3,4} = {2-Z}{1.2pt,white},
            vline{3,4} = {2-Z}{1.2pt,white},
        }
            \toprule
            &
            \| \Phi_{t}(x(t))\|
            &
            \frac{1}{t} \int_{t_{0}}^{t} \langle x(s) - x^{*}, V(x(s))\rangle ds
            &
            h(x_{t}^{\text{\textnormal{best}}}) - h(x^{*})
            \\
            \midrule
            \delta>1
            &
            \mathcal{O} \left(\frac{1}{\sqrt{t}}\right)
            &
            \mathcal{O} \left(\frac{1}{t}\right)
            &
            \text{--}
            \\
            \delta=1
            &
            \mathcal{O} \left(\sqrt{\frac{\ln t}{t}}\right)
            &
            \mathcal{O} \left(\frac{\ln t}{t}\right)
            &
            \leq
            \mathcal{O} \left(\frac{1}{\ln t}\right)
            \\
            0 < \delta < 1 
            &
            \mathcal{O}\left(\frac{1}{t^{\frac{\delta}{2}}}\right)
            &
            \mathcal{O}\left(\frac{1}{t^{\delta}}\right)
            &
            \leq \mathcal{O}\left(\frac{1}{t^{1 - \delta}}\right)
            \\
            \bottomrule
        \end{tblr}
    \end{whitetablebox}
    where we write $\leq$ to emphasize that we only obtain an upper bound for the sign-less quantity $h(x_{t}^{\textnormal{best}}) - h(x^{*})$. 
\end{theorem}
\begin{proof}
    Choose $0 < \lambda < \frac{1}{\delta}$. After integrating \eqref{eq:descent-property-E-lambda-statement} from $t_{0}$ to $t\geq t_{0}$, using the definition of $x_{t}^{\text{best}}$, applying Jensen's inequality to $h$ and then plugging \eqref{eq:lower-bound-E-lambda-statement}, we arrive at
    \begin{align*}
        &\rph{\leq} \frac{\lambda t}{2} \| \Phi_{t}(x(t))\|^{2} + \left( \zeta_{1}(t) + \int_{t_{0}}^{t} \zeta_{2}(s) ds\right) (\inf h - h(x^{*})) \\
        &\leq \frac{\lambda t}{2} \| \Phi_{t}(x(t))\|^{2} + \left( \zeta_{1}(t) + \int_{t_{0}}^{t} \zeta_{2}(s) ds\right) (h(x_{t}^{\text{best}}) - h(x^{*})) \\
        &\leq \frac{\lambda t}{2} \| \Phi_{t}(x(t))\|^{2} + \zeta_{1}(t) (h(x(t)) - h(x^{*})) + \left(\int_{t_{0}}^{t} \zeta_{2}(s) ds\right) (h(\bar{x}_{t}) - h(x^{*})) \\
        &\leq \frac{\lambda t}{2} \| \Phi_{t}(x(t))\|^{2} + \zeta_{1}(t) (h(x(t)) - h(x^{*})) + \int_{t_{0}}^{t} \zeta_{2}(s) (h(x(s)) - h(x^{*})) ds \\
        &\leq \cE_{\lambda}(t) + \int_{t_{0}}^{t} \zeta_{2}(s) (h(x(s)) - h(x^{*})) ds \leq - (\lambda + 1) \int_{t_{0}}^{t} \langle x(s) - x^{*}, V(x(s))\rangle ds + \cE_{\lambda}(t_{0}),
    \end{align*}
    which produces two inequalities. One the one hand, after rearranging terms and dividing by $t$, we get 
    \begin{equation}\label{eq:ergodic-rate-cont-1}
        \frac{\lambda}{2} \| \Phi_{t}(x(t))\|^{2} + \frac{\lambda + 1}{t}\int_{t_{0}}^{t} \langle x(s) - x^{*}, V(x(s))\rangle ds \leq \frac{\cE_{\lambda}(t_{0})}{t} + \frac{1}{t}\left( \zeta_{1}(t) + \int_{t_{0}}^{t} \zeta_{2}(s)ds\right) (h(x^{*}) - \inf h).
    \end{equation}
    On the other hand, after dropping the norm squared term on the left-hand side and the nonpositive integral term on the right-hand side we also obtain 
    \begin{equation*}
        \left( \zeta_{1}(t) + \int_{t_{0}}^{t} \zeta_{2}(s)ds\right) (h(x_{t}^{\textnormal{best}}) - h(x^{*})) \leq \cE_{\lambda}(t_{0}).
    \end{equation*}
    The stated rates follow after recalling that $\zeta_{1}(t) \asymp \zeta_{2}(t) \asymp \varepsilon(t) \asymp \frac{1}{t^{\delta}}$ as $t \to +\infty$, and thus 
    \begin{equation}
        \zeta_{1}(t) + \int_{t_{0}}^{t} \zeta_{2}(s)ds \asymp
        \begin{dcases}
            \frac{1}{t^{\delta}} + \frac{1}{\delta - 1} \left( \frac{1}{t_{0}^{\delta - 1}} - \frac{1}{t^{\delta - 1}}\right) & \text{if } \delta > 1, \\
            \frac{1}{t} + \ln(t) - \ln(t_{0}) & \text{if } \delta = 1, \\
            \frac{1}{t^{\delta}} + \frac{1}{1 - \delta}  (t^{1 - \delta} - t_{0}^{1 - \delta}) & \text{if } 0 < \delta < 1
        \end{dcases} \ \mbox{as} \quad t \to +\infty.
    \end{equation}
\end{proof}

\begin{remark}
    Since $\| \Phi_{t}(x(t))\| = \| V(x(t)) + \varepsilon(t) \nabla h(x(t))\|$, we can obtain further rates for $\| V(x(t))\|$ if we assume that $t\mapsto x(t)$ is bounded. Indeed, from the previous theorem and the triangle inequality, we get
    \[
        \| V(x(t))\| = 
        \begin{dcases}
            \mathcal{O}\left( \frac{1}{\sqrt{t}}\right) &\text{if } \delta > 1, \\
            \mathcal{O}\left( \sqrt{\frac{\ln t}{t}}\right) &\text{if } \delta = 1,  \\
            \mathcal{O}\left( \frac{1}{t^{\frac{\delta}{2}}}\right) &\text{if } 0 < \delta < 1
        \end{dcases}\ \mbox{as} \quad t \to +\infty.
    \]
\end{remark}

\subsection{Geometric assumption on \texorpdfstring{$V$}{}}
Briefly recall the simple bilevel optimization problem  
\begin{equation}\label{eq:bilevel-where-inner-is-minimization}
    \begin{aligned}
        &\text{minimize} & &h(x) \\
        &\text{subject to } & & x\in \argmin f,
    \end{aligned}
\end{equation}
where $f : \mathcal{H}\to \R$ is also a convex and continuously Fr\'echet differentiable function, which is a particular case of \ref{eq:bilevel-where-inner-is-monotone-equation} with $V = \nabla f$. In this setting, we first considered a weaker integrability condition due to Attouch and Czarnecki: for every optimal solution $x^{*}$ to \eqref{eq:bilevel-where-inner-is-minimization} and every $p \in N_{\argmin f}(x^{*})$, we assumed that
\[
    \int_{t_{0}}^{+\infty} \Bigl[(f - \min f)^{*}(\varepsilon(t) p) - \sigma_{\argmin f}(\varepsilon(t) p)\Bigr] dt < +\infty.
\]
Second, we imposed a stronger H\"olderian error bound condition on the inner function: for some constant for some constant $\tau > 0$ and some exponent $\rho\in (1, 2]$, we assumed that
\[
   \frac{\tau}{\rho} \dist(x, \argmin f)^{\rho} \leq f(x) - \min f \quad \forall x \in \cH.
\]
In our smooth setting, the role of $(f - \min f)^{*}$ will be played instead by the \emph{Fitzpatrick function} of $V$. Recall that for a general monotone (and possibly multi-valued) operator $T : \cH \to 2^{\cH}$, its Fitzpatrick function is given by 
\begin{equation}\label{eq:def-Fitzpatrick}
    \mathcal{F}_{T}(x, u) := \sup_{(y, v) \in \graph T} \bigl\{ \langle y, u\rangle + \langle x, v\rangle - \langle y, v\rangle\bigr\}.
\end{equation}
An important fact is that, for every $x\in \cH$, by taking the supremum over $(y, 0)$, where $y\in \zer T$, we obtain
\begin{equation}\label{eq:fitzpatrick-support-function-inequality}
    \mathcal{F}_{T}(x, u) \geq \sup_{y \in \zer T} \bigl\{ \langle y, u\rangle + \langle x, 0\rangle - \langle y, 0\rangle\bigr\} = \sup_{y\in \zer T} \langle y, u\rangle = \sigma_{\zer T}(u). 
\end{equation} 
In order to get stronger convergence guarantees for our system, we will make the following geometric assumption on $V$, which has been introduced in \cite{BotCsetnek2016JMAA} and can be seen as the analogue of the Attouch-Czarnecki integrability condition.

\begin{ass}\label{ass:Attouch-Czarnecki}
    We say that the operator $V$ satisfies a \emph{Attouch--Czarnecki-type} condition if for every optimal solution $x^{*}$ to \eqref{eq:bilevel-where-inner-is-monotone-equation} and every $p^*\in N_{\zer V}(x^{*})$, it holds 
    \[
        \int_{t_{0}}^{+\infty} \Bigl[ \mathcal{F}_{V}(x^{*}, \varepsilon(t) p) - \sigma_{\zer V}(\varepsilon(t) p)\Bigr] dt < +\infty. 
    \]
\end{ass} 
\noindent The condition is well defined, since the integrand is nonnegative by \eqref{eq:fitzpatrick-support-function-inequality}. Since $V$ is maximally monotone, the set $\zer V$ is convex and closed, and $N_{\zer V}(x^{*}) :=\{p \in \cH : \langle p, x-x^* \rangle \leq 0 \ \forall x \in \zer V\}$ denotes the normal cone to $\zer V$ at $x^*$.

Interestingly, we will show that this integrability condition will be enough for establishing full convergence guarantees (that is, convergence rates for both the inner and outer residuals as well as weak convergence of the trajectories towards optimal solutions to \eqref{eq:bilevel-where-inner-is-monotone-equation}) without having to rely on a sharpness or Hölder-type assumption for the inner level. The following lemma will collect some essential results related to Assumption \ref{ass:Attouch-Czarnecki}.
\begin{lemma}\label{lem:assumptions}
    Let $x^{*}$ be a solution to \eqref{eq:bilevel-where-inner-is-monotone-equation} and let $p \in N_{\zer V}(x^{*})$, and define $\varphi_{x^{*}, p} : [0, +\infty) \to \bar{\R}$ by
    \[
        \varphi_{x^{*}, p}(s) := \mathcal{F}_{V}(x^{*}, s p) - \sigma_{\zer V}(s p). 
    \]
    Then, the following hold: 
    \begin{enumerate}[\rm (i)]
        \item $\varphi_{x^{*}, p}(s) \geq 0$ for every $s \geq 0$ and $\varphi_{x^{*}, p}(0) = 0$; 
        \item $\varphi_{x^{*}, p}$ is convex;
        \item $\varphi_{x^{*}, p}$ is monotonically increasing on $[0, +\infty)$; 
        \item If $p^{*} = -\nabla h(x^{*})$, then for every $x \in \cH$ and every $s \geq 0$ we have 
        \[
            - \langle x - x^{*}, V(x)\rangle - s ( h(x) - h(x^{*})) \leq \varphi_{x^{*}, p^{*}} (s) \text{;}
        \]
        \item If $V$ satisfies Assumption \ref{ass:Attouch-Czarnecki}, then $\varphi_{x^{*}, p}(\varepsilon(t)) = o\left(\frac{1}{t}\right)$ as $t\to +\infty$. 
    \end{enumerate}
\end{lemma}
\begin{proof}
    (i) According to \eqref{eq:fitzpatrick-support-function-inequality}, we know that $\varphi_{x^{*}, p}(s) \geq 0$ for every $s\geq 0$. Furthermore, noting that $\sigma_{\zer V}(sp) = \langle x^{*}, s p\rangle$ (on account of $p \in N_{\zer V}(x^{*})$), we have 
    \[
        \varphi_{x^{*}, p}(0) = \sup_{y\in \cH} \{ \langle x^{*}, V(y)\rangle - \langle y, V(y)\rangle\} = \sup_{y\in \cH} - \langle y - x^{*}, V(y)\rangle = 0, 
    \]
    since $\langle y - x^{*}, V(y)\rangle \geq 0$ for every $y\in \cH$ and the inequality is saturated when $y = x^{*}$. 

    (ii) We write 
    \[
        \varphi_{x^{*}, p}(s) = \sup_{y\in \cH} \bigl\{ \langle y, sp\rangle + \langle x^{*}, V(y)\rangle - \langle y, V(y)\rangle\bigr\} - \langle x^{*}, sp\rangle = \sup_{y\in \cH} \bigl\{ s \langle y - x^{*}, p\rangle + \langle x^{*} - y, V(y)\rangle\bigr\}, 
    \]
    i.e., $\varphi_{x^{*}, p}$ is a pointwise supremum of affine functions in $s$, thus convex. 

    (iii) Take $0 \leq s_{1} < s_{2}$. Then, using the convexity of $\varphi_{x^{*}, p}$ together with $\varphi_{x^{*}, p}(0) = 0$ gives
    \begin{equation*}
        \varphi_{x^{*}, p}(s_{1}) = \varphi_{x^{*}, p}\left( \left(1 - \frac{s_{1}}{s_{2}}\right)\cdot 0 + \frac{s_{1}}{s_{2}} \cdot s_{2}\right) \leq \left(1 - \frac{s_{1}}{s_{2}}\right) \varphi_{x^{*}, p}(0) + \frac{s_{1}}{s_{2}} \varphi_{x^{*}, p}(s_{2}) = \frac{s_{1}}{s_{2}} \varphi_{x^{*}, p}(s_{2}) \leq \varphi_{x^{*}, p}(s_{2}). 
    \end{equation*}

    (iv) We use the fact that $\sigma_{\zer V}(s p^{*}) = \langle x^{*}, sp^{*}\rangle$, together with the gradient inequality applied to $h$ and recall the definition \eqref{eq:def-Fitzpatrick} of the Fitzpatrick function of $V$ to obtain, for every $x \in \cH$, 
    \begin{align*}
        -\langle x - x^{*}, V(x)\rangle - s(h(x) - h(x^{*})) &\leq - \langle x - x^{*}, V(x)\rangle - s \langle x - x^{*}, \nabla h(x^{*})\rangle \\
        &= \langle x, s p^{*}\rangle + \langle x^{*}, V(x)\rangle - \langle x, V(x)\rangle - \langle x^{*}, s p^{*}\rangle \\
        &= \langle x, s p^{*}\rangle + \langle x^{*}, V(x)\rangle - \langle x, V(x)\rangle - \sigma_{\zer V}(s p^{*}) \\
        &\leq \sup_{y\in \cH} \Bigl\{ \langle y, s p^{*}\rangle + \langle x^{*}, V(y)\rangle - \langle y, V(y)\rangle\Bigr\} - \sigma_{\zer V}(s p^{*}) \nonumber \\
        &= \mathcal{F}_{V}(x^{*}, s p^{*}) - \sigma_{\zer V}(s p^{*}). 
    \end{align*}

    (v) Assumption \ref{ass:Attouch-Czarnecki} means that $\int_{t_{0}}^{+\infty} \varphi_{x^{*}, p}(\varepsilon(t)) dt < +\infty$. For $t \geq 2t_{0}$, we focus on the interval $[t/2, t]$. Since $\varepsilon$ is decreasing, we have $\varepsilon(t) \leq \varepsilon(s)$ for all $s \in [t/2, t]$. Since $\varphi_{x^{*}, p}$ is monotonically increasing, this gives $\varphi_{x^{*}, p} (\varepsilon(t)) \leq \varphi_{x^{*}, p}(\varepsilon(s))$ for every $s \in [t/2, t]$, which entails
    \[
        \frac{t}{2} \cdot \varphi_{x^{*}, p}(\varepsilon(t)) \leq \int_{t/2}^{t} \varphi_{x^{*}, p}(\varepsilon(s)) ds = \int_{t_{0}}^{t} \varphi_{x^{*}, p}(\varepsilon(s)) ds - \int_{t_{0}}^{t/2} \varphi_{x^{*}, p}(\varepsilon(s)) ds \to 0
    \]
    as $t\to +\infty$. 
\end{proof}
\begin{remark}\label{rem:varphi-multiple-invariant}
    For any $p\in N_{\zer V}(x^{*})$, we also have $\eta p \in N_{\zer V}(x^{*})$ for $\eta \geq 0$, which means that if Assumption \ref{ass:Attouch-Czarnecki} holds, we have $\int_{t_{0}}^{+\infty} \varphi_{x^{*}, p}(\eta (\varepsilon(t))) dt = \int_{t_{0}}^{+\infty} \varphi_{x^{*}, \eta p} (\varepsilon(t)) dt < +\infty$. 
\end{remark}

\begin{remark}\label{rem:sharpness-cont}
    As we will see, this Attouch--Czarnecki integrability condition will be enough to establish last-iterate (or rather, last-time) convergence statements for system \eqref{eq:first-order-system-regularized-operator}. However, it is worth noting that Assumption \ref{ass:Attouch-Czarnecki} is implied by other type of condition which usually appears in the literature; namely, a \emph{sharpness} assumption which plays an analogous role to the Hölderian error bound condition in the optimization setting. Such an assumption reads as follows: the operator $V$ is said to satisfy a sharpness condition with exponent $\rho \in (1, 2)$ if for every optimal solution $x^{*}$ to \eqref{eq:bilevel-where-inner-is-monotone-equation} there exists some constant $\tau > 0$ such that for every $x\in \cH$
    \begin{equation}\label{eq:sharpness-smooth}
        \frac{\tau}{\rho} \dist(x, \zer V)^{\rho} \leq \langle x - x^{*}, V(x)\rangle.
    \end{equation}
    Such condition implies Assumption \ref{ass:Attouch-Czarnecki}; indeed, by writing the $\dist(x, \zer V)^{\rho}$ as an infimal convolution, the previous inequality is equivalent to  
    \[
        \langle x, u\rangle - \langle x - x^{*}, V(x)\rangle \leq \langle x, u\rangle - \left(\iota_{\zer V} \infconv \frac{\tau}{\rho} \| \cdot\|^{\rho}\right)(x) 
    \]
    for every $x, u \in \cH$. After taking the supremum over $x\in \cH$ on both sides of the previous inequality and recalling the definition of the Fitzpatrick function of $V$ as well as elementary properties of the Fenchel conjugate, we come to 
    \[
        \mathcal{F}_{V}(x^{*}, u) \leq \left( \iota_{\zer V} \infconv \frac{\tau}{\rho} \| \cdot\|^{\rho}\right)^{*}(u) = \iota_{\zer V}^{*}(u) + \left(\frac{\tau}{\rho} \| \cdot\|^{\rho}\right)^{*}(u) = \sigma_{\zer V}(u) + \frac{\tau^{1 - \rho^{*}}}{\rho^{*}} \| u\|^{\rho^{*}},
    \]
    where $\rho^{*}$ is such that $\frac{1}{\rho} + \frac{1}{\rho^{*}} = 1$, and thus 
    \[
        \mathcal{F}_{V} (x^{*}, u) - \sigma_{\zer V}(u) \leq \frac{\tau^{1 - \rho^{*}}}{\rho^{*}} \| u\|^{\rho^{*}}, 
    \]
    which means that \eqref{eq:sharpness-smooth} implies Assumption \ref{ass:Attouch-Czarnecki} whenever $\delta > \frac{1}{\rho^{*}}$.
\end{remark}

\begin{example}\label{ex:AC-without-sharpness}
    The continuous and monotone operator
    \[
        V:\mathbb R\to\mathbb R,
        \qquad
        V(x)=x^3.
    \]
    provides an example of an operator that satisfies Assumption \ref{ass:Attouch-Czarnecki} for a suitable range of the
    decay exponent $\delta$, while it does not satisfy the sharpness
    condition~\eqref{eq:sharpness-smooth} for any $\rho\in(1,2)$.
    
    Since $\zer V=\{0\}$,
    we have
    \[
        N_{\zer V}(0)=\mathbb R
        \qquad\text{and}\qquad
        \sigma_{\zer V}(u)=0
        \quad \forall u\in\mathbb R.
    \]
    Moreover, by the definition of the Fitzpatrick function,
    \[
    \begin{aligned}
        \mathcal{F}_V(0,u)
        &=\sup_{y\in\mathbb R}
        \big\{yu-yV(y)\big\} =\sup_{y\in\mathbb R}
        \big\{yu-y^4\big\} = \frac{3}{4^\frac{4}{3}}|u|^\frac{4}{3} \quad \forall u\in\mathbb R.
    \end{aligned}
    \]
    Thus, for every $p\in N_{\zer V}(0)=\mathbb R$ and every $t \geq t_0$,
    \[
        \mathcal{F}_V(0,\varepsilon(t)p) -\sigma_{\zer V}(\varepsilon(t)p)
        =\frac{3}{4^\frac{4}{3}}
        |\varepsilon(t)p|^\frac{4}{3}
        =\frac{3c^\frac{4}{3}|p|^\frac{4}{3}}{4^\frac{4}{3}}
          \frac{1}{t^\frac{4 \delta}{3}}.
    \]
    It follows that
    \[
        \int_{t_0}^{+\infty}
        \left[
            \mathcal{F}_V(0,\varepsilon(t)p)
            -\sigma_{\zer V}(\varepsilon(t)p)
        \right]dt
        <+\infty
    \]
    whenever
    \[
        \delta>\frac34.
    \]
    Hence $V$ satisfies Assumption~2 for every $\delta>\frac34$.
    
    On the other hand, $V$ does not satisfy the sharpness
    condition~\eqref{eq:sharpness-smooth} for any $\rho\in(1,2)$. Indeed, since
    $\zer V=\{0\}$, such a condition would require the existence of
    $\tau>0$ such that
    \[
        \frac{\tau}{\rho}|x|^\rho
        \leq
        \langle x,V(x)\rangle
        =|x|^4
        \qquad \forall x\in\mathbb R.
    \]
    Thus no $\tau>0$ can satisfy the above inequality.
\end{example}

\subsection{Convergence analysis under Assumption \ref{ass:Attouch-Czarnecki}}
The following proposition provides two key results. 
\begin{proposition}\label{prop:lim-trajectories-and-operator}
 Let $x : [t_{0}, +\infty) \to \cH$ be a solution trajectory to \eqref{eq:first-order-system-regularized-operator} and let $x^{*}$ be an optimal solution to \eqref{eq:bilevel-where-inner-is-monotone-equation}. Assume that $V$ satisfies Assumption \ref{ass:Attouch-Czarnecki}. Then, for every $0 < \lambda < \frac{1}{\delta}$, it holds
    \[
        \lim_{t\to +\infty} \cE_{\lambda}(t) \ \text{exists}.
    \]
Consequently, we obtain 
    \[
        \lim_{t\to +\infty} t \| \Phi_{t}(x(t))\|^{2} = 0 \quad \text{and} \quad \lim_{t\to +\infty} \| x(t) - x^{*}\| \ \text{exists}.
    \]
\end{proposition}
\begin{proof}
    We begin from \eqref{eq:descent-property-E-lambda-statement} in Lemma \ref{lem:full-energy-cont}. It is straightforward to check that $\zeta_{2}(t) \leq (1 + \lambda)\varepsilon(t)$ for every $t \geq t_0$. Let $p^{*} := -\nabla h(x^{*}) \in N_{\zer V}(x^{*})$. After rearranging terms and using Lemma \ref{lem:assumptions}, for every $t \geq t_0$, we come to 
    \begin{align}
        &\rph{\leq}\dot{\cE}_{\lambda}(t) + \bigl\| \dot{x}(t)\bigr\|^{2} + \frac{\lambda}{2} \| \Phi_{t}(x(t))\|^{2} + \frac{1 + \lambda}{2} \langle x(t) - x^{*}, V(x(t))\rangle \nonumber\\
        &\leq - \frac{1}{2} \Bigl[ (1 + \lambda) \langle x(t) - x^{*}, V(x(t))\rangle + 2\zeta_{2}(t) (h(x(t)) - h(x^{*}))\Bigr] \nonumber\\
        &\leq 
        \begin{dcases}
            0 &\text{if }h(x(t)) \geq h(x^{*})\\
            - \frac{1}{2} \Bigl[ \langle x(t) - x^{*}, V(x(t))\rangle + 2 (1 + \lambda)\varepsilon(t) (h(x(t)) - h(x^{*}))\Bigr] &\text{if }h(x(t)) < h(x^{*})
        \end{dcases} \nonumber\\
        &\leq \frac{1}{2} \varphi_{x^{*}, p^{*}} (2(1 + \lambda) \varepsilon(t)). \label{eq:lim-E-lambda-1}
    \end{align}
    According to Assumption \ref{ass:Attouch-Czarnecki} and Remark \ref{rem:varphi-multiple-invariant}, the right-hand side of the previous inequality is integrable over $[t_{0}, +\infty)$, which in particular means that $t\mapsto \cE_{\lambda}(t)$ is upper bounded. From \eqref{eq:lower-bound-E-lambda-statement} in Lemma \ref{lem:full-energy-cont}, for every $t \geq t_0$, we have 
    \[
        \cE_{\lambda}(t) \geq \frac{1}{2} \| x(t) - x^{*}\|^{2} + \frac{\lambda t}{2} \| \Phi_{t}(x(t))\|^{2} + \zeta_{1}(t) (\inf h - h(x^{*})).
    \]
    Since $\zeta_{1}(t) \asymp \varepsilon(t) \to 0$ as $t \to +\infty$, we also obtain that $t\mapsto \| x(t) - x^{*}\|$ and $t\mapsto t \| \Phi_{t}(x(t))\|^{2}$ are upper bounded and that $t\mapsto \cE_{\lambda}(t)$ is lower bounded, and together with the integrability of the right-hand side of \eqref{eq:lim-E-lambda-1}, this implies that 
    \[
        \lim_{t\to +\infty} \cE_{\lambda}(t) \quad \text{exists.}
    \]
    Plugging this back into \eqref{eq:lim-E-lambda-1} gives 
    \begin{equation}
        \int_{t_{0}}^{+\infty} \bigl\| \dot{x}(t)\bigr\|^{2} dt < +\infty, \quad \int_{t_{0}}^{+\infty} \| \Phi_{t}(x(t))\|^{2} dt < +\infty \quad \text{and} \quad \int_{t_{0}}^{+\infty} \langle x(t) - x^{*}, V(x(t))\rangle dt < +\infty. \label{eq:lim-E-lambda-2}
    \end{equation}
    Now, choose $0 < \lambda_{1} < \lambda_{2} < \frac{1}{\delta}$. Since 
    \[
        \cE_{\lambda_{2}}(t) - \cE_{\lambda_{1}}(t) = (\lambda_{2} - \lambda_{1}) \cE_{2}(t) \quad \forall t \geq t_0,  
    \]
    we deduce that $\lim_{t\to +\infty} \cE_{2}(t)$ exists, thus also $\lim_{t\to +\infty} \cE_{1}(t)$ exists. From $\lim_{t \to +\infty} \| \Phi_{t}(x(t))\| = 0$, which holds according to Theorem \ref{thm:ergodic-rate-cont}, due the boundedness of $t\mapsto x(t)$ and the fact that $\nabla h$ (and thus also $h$) is bounded on bounded subsets of $\cH$, we obtain that $\lim_{t\to +\infty} \| V(x(t))\| = 0$; this in turn ensures that the summands in $\cE_{1}(t)$ other that $\frac{1}{2} \| x(t) - x^{*}\|^{2}$ vanish, thus
    \[
        \lim_{t\to +\infty} \| x(t) - x^{*}\| \quad \text{exists.}
    \]
    The previous statment, together with the existence of $\lim_{t\to +\infty} \cE_{2}(t)$, entails
    \[
        \lim_{t\to +\infty} t \| \Phi_{t}(x(t))\|^{2} \quad \text{exists.}
    \]
    Moreover, according to \eqref{eq:lim-E-lambda-2}, we have 
    \[
        \int_{t_{0}}^{+\infty} \frac{t \| \Phi_{t}(x(t))\|}{t} dt = \int_{t_{0}}^{+\infty} \| \Phi_{t}(x(t))\|^{2} dt < +\infty, 
    \]
    and since $\int_{t_{0}}^{+\infty} \frac{1}{t} dt = +\infty$, the only value the limit can take is 
    \[
        \lim_{t\to +\infty} t\| \Phi_{t}(x(t))\|^{2} = 0. 
    \]
\end{proof} 
We are now in a position to state and prove the main theorem of this section.
\begin{theorem}\label{thm:continuous-time-full-rates}
    Let $x : [t_{0}, +\infty) \to \mathcal{H}$ be a solution trajectory to \eqref{eq:first-order-system-regularized-operator} and let $x^{*}$ be an optimal solution to \eqref{eq:bilevel-where-inner-is-monotone-equation}. Assume that $V$ satisfies Assumption \ref{ass:Attouch-Czarnecki}. Then, the following statements hold:
    \begin{enumerate}[\rm (i)]
        \item If $\delta > \frac{1}{2}$, then $\| V(x(t))\| = o\left(\frac{1}{\sqrt{t}}\right)$ as $t \to +\infty$.
        \item If $0 < \delta \leq \frac{1}{2}$, then
        \[
            \| V(x(t))\| = \mathcal{O}\left(\frac{1}{t^{\delta}}\right), \quad
            \langle x(t) - x^{*}, V(x(t))\rangle = o\left(\frac{1}{\sqrt{t}}\right) \quad \text{and} \quad |h(x(t)) - h(x^{*})| = o\left(\frac{1}{t^{\frac{1}{2} - \delta}}\right)
        \]
        as $t \to +\infty$. Furthermore, $x(t)$ converges weakly to an optimal solution to the bilevel problem \eqref{eq:bilevel-where-inner-is-monotone-equation} as $t \to +\infty$.
    \end{enumerate}
\end{theorem}
\begin{proof}
    Throughout the proof, we will be invoking the statements given by Proposition \ref{prop:lim-trajectories-and-operator}. In both cases, we know that $\| \Phi_{t}(x(t))\| = o\left(\frac{1}{\sqrt{t}}\right)$ as $t\to +\infty$. If $\delta > \frac{1}{2}$, i.e., in case (i), this implies that $\|V(x(t))\| = o\left(\frac{1}{\sqrt{t}}\right)$ as $t\to +\infty$. When $\delta \leq \frac{1}{2}$, i.e., in case (ii), we get $\|V(x(t))\| = \mathcal{O}(\varepsilon(t)) = \mathcal{O}\left(\frac{1}{t^{\delta}}\right)$ as $t\to +\infty$ instead. Moreover, by using the gradient inequality on $h$, for every $t\geq t_{0}$, we obtain
    \begin{equation}\label{eq:small-o-V-h-cont-1}
        \begin{aligned}
            2\langle x(t) - x^{*}, V(x(t))\rangle + 2\varepsilon(t) (h(x(t)) - h(x^{*})) &\leq 2\langle x(t) - x^{*}, \Phi_{t}(x(t))\rangle \\
            &\leq 2\| x(t) - x^{*}\| \| \Phi_{t}(x(t))\| = o\left(\frac{1}{\sqrt{t}}\right)
        \end{aligned}
    \end{equation}
    as $t\to +\infty$ due to the boundedness of $t\mapsto x(t)$. Morevoer, according to Lemma \ref{lem:assumptions}, we also have 
    \begin{equation}\label{eq:small-o-V-h-cont-2}
        - \langle x(t) - x^{*}, V(x(t))\rangle - 2\varepsilon(t) (h(x(t)) - h(x^{*})) \leq \varphi_{x^{*}, p^{*}}(2\varepsilon(t)) = o\left(\frac{1}{t}\right)
    \end{equation}
    as $t\to +\infty$. Summing \eqref{eq:small-o-V-h-cont-1} and \eqref{eq:small-o-V-h-cont-2} gives 
    \[
        \langle x(t) - x^{*}, V(x(t))\rangle = o\left(\frac{1}{\sqrt{t}}\right) + o\left(\frac{1}{t}\right) = o\left(\frac{1}{\sqrt{t}}\right) 
    \]
    as $t\to +\infty$. Dropping the nonnegative term from \eqref{eq:small-o-V-h-cont-1}, and then multiplying \eqref{eq:small-o-V-h-cont-1} by $\frac{1}{2}$ and adding it to \eqref{eq:small-o-V-h-cont-2}, respectively give
    \[
        \varepsilon(t) (h(x(t)) - h(x^{*})) \leq o\left(\frac{1}{\sqrt{t}}\right) \quad \text{and} \quad - \varepsilon(t) (h(x(t)) - h(x^{*})) \leq o\left(\frac{1}{\sqrt{t}}\right)
    \]
    and thus 
    \[
        | h(x(t)) - h(x^{*})| = o\left(\frac{1}{t^{\frac{1}{2} - \delta}}\right)
    \]
    as $t\to +\infty$. Regarding weak convergence of trajectories, we know that for any optimal solution $x^{*}$ to \eqref{eq:bilevel-where-inner-is-monotone-equation}, the limit $\lim_{t\to +\infty} \| x(t) - x^{*}\|$ exists, which verifies the first condition of Opial's Lemma. Now, let $\bar{x}$ be a weak sequential cluster point of $t \mapsto x(t)$, and let $(t_{k})_{k\geq 1} \subseteq [t_{0}, +\infty)$ be a strictly increasing sequence such that $t_{k} \to +\infty$ and $x(t_{k}) \rightharpoonup \bar{x}$ (i.e., weakly) as $k\to +\infty$. We know that $V(x(t_{k})) \to 0$ as $k\to +\infty$. Since $\graph V$ is closed in $\cH^{\text{weak}}\times \cH^{\text{strong}}$, we obtain that $V(\bar{x}) = 0$, so $\bar{x}$ is feasible for problem \eqref{eq:bilevel-where-inner-is-monotone-equation}. To show that it is also optimal, we use the (weak) lower semicontinuity of $h$, and the fact that $\lim_{t\to +\infty} h(x(t)) = h(x^{*})$. We have 
    \[
        h(\bar{x}) \leq \liminf_{k\to +\infty} h(x(t_{k})) = \lim_{k\to +\infty} h(x(t_{k})) = h(x^{*}). 
    \]
    This shows that $\bar{x}$ is also optimal for problem \eqref{eq:bilevel-where-inner-is-monotone-equation}, which verifies the second condition in Opial's Lemma and thus concludes the proof for this theorem.
\end{proof}

\section{A regularized proximal-extragradient algorithm for the composite smooth+nonsmooth bilevel optimization problem}
We now extend \eqref{eq:bilevel-where-inner-is-monotone-equation} by incorporating nonsmooth components at both the inner and outer levels, leading to the bilevel optimization problem
\begin{equation}\label{eq:bilevel-where-inner-is-monotone-inclusion}
    \begin{aligned}
        &\text{minimize} & & H(x) := h(x) + \hat{h}(x). \\
        &\text{subject to } & &0 \in V(x) + \partial \hat{f}(x)
    \end{aligned}
\end{equation}
Throughout this section, we make the following standing assumption on our problem. 
\begin{ass}\label{ass:standing-disc}
$h : \cH \to \R$ is convex, continuously Fréchet differentiable and $\nabla h$ is $L_{\nabla h}$-Lipschitz continuous for some $L_{\nabla h} > 0$. The operator $V : \cH \to \cH$ is monotone and $L_{V}$-Lipschitz continuous for some $L_{V} > 0$. The functions $\hat{f}, \hat{h} : \cH \to \bar{\R}$ are proper, convex and lower semicontinuous. Additionally, we make the following regularity assumptions: 
    \begin{enumerate}[\rm (i)]
        \item For every $k\geq 0$, $\partial (\hat{f} + \varepsilon_{k} \hat{h}) = \partial \hat{f} + \varepsilon_{k} \partial \hat{h}$; 
        \item For every optimal solution $x^{*}$ to \eqref{eq:bilevel-where-inner-is-monotone-inclusion}, it holds $0 \in \partial H(x^{*}) + N_{\zer(V + \partial \hat{f})}(x^{*})$. 
    \end{enumerate}
Finally, we require $\partial \hat{h}$ to map bounded subsets of its domain to bounded subsets of $\cH$. 
\end{ass}

In the convex analysis literature (see, for instance, \cite{BauschkeCombettes2017,Bot2010}), various regularity conditions are available that guarantee the exact subdifferential sum formulas in (i) and (ii) above.

\subsection{The algorithm}
With a motivation similar to that of the continuous-time approach, we propose a diagonal algorithm that, at each iteration $k \geq 0$ will performs one step of the proximal-extragradient method applied to the problem of finding a zero of the sum of the dynamically regularized operators
\[
    \Phi_{k}(x) := V(x) + \varepsilon_{k} \nabla h(x) \quad \mbox{and} \quad \partial g_{k}(x) = \partial \hat{f}(x) + \varepsilon_{k} \partial \hat{h}(x),  
\]
where 
\[
    g_{k}(x) := \hat{f}(x) + \varepsilon_{k} \hat{h}(x)
\]
and 
\[
    \varepsilon_{k} := \frac{c}{(k + a)^{\delta}}, \ \mbox{with} \ c, a, \delta > 0.
\]
We that the step size satisfies $s \in \left(0, \frac{1}{L_V}\right)$, which ensures that, for sufficiently large $k \geq 0$, $s \in \left(0, \frac{1}{L_k}\right]$, where  
\[
    L_{k} := L_{V} + \varepsilon_{k} L_{\nabla h} \quad \forall k \geq 0.
\]

\begin{algorithm}[t]
    \caption{Regularized proximal-extragradient method (Bi-EG)}
    \label{alg:fb-reg-eg}
    \begin{algorithmic}[1]
        \Require Initial point $x_{0} \in \cH$, stepsize $s \in \left(0, \frac{1}{L_V}\right)$
        \For{$k = 0,1,2,\dots$}
            \State $y_k := \prox_{s  \hat{f} + s\varepsilon_{k} \hat{h}} \big(x_k - sV(x_k) - s \varepsilon_k \nabla h(x_k) \big)$
            \State $x_{k+1} := \prox_{s  \hat{f} + s\varepsilon_{k} \hat{h}} \big(x_k - sV(y_k) - s \varepsilon_k \nabla h(y_k) \big)$
        \EndFor
    \end{algorithmic}
\end{algorithm} 
\noindent The optimality measure for the inner level will be given by 
\begin{equation}\label{eq:definition-inner-object}
    F_{x^{*}}(x) := \langle x - x^{*}, V(x)\rangle + \hat{f}(x) - \hat{f}(x^{*}), 
\end{equation}
where $x^{*}$ is an optimal solution to the bilevel problem \eqref{eq:bilevel-where-inner-is-monotone-inclusion}. 
\begin{remark}\label{rem:properties-F}
    The inner object $F_{x^{*}}$ is indeed nonnegative. Indeed, every $x^{*} \in \zer (V + \partial \hat{f})$ is such that $ - V(x^{*}) \in \partial \hat{f}(x^{*})$. It follows that, for every $x \in \cH$,
    \begin{align*}
        \langle x - x^{*}, V(x)\rangle &= \langle x - x^{*}, V(x) - V(x^{*})\rangle + \langle x^{*} - x, -V(x^{*})\rangle \\
        &\geq \langle x - x^{*}, V(x) - V(x^{*})\rangle - (\hat{f}(x) - \hat{f}(x^{*}))\\
        & \geq - (\hat{f}(x) - \hat{f}(x^{*})).
    \end{align*}
    Furthermore, we have $F_{x^{*}}(y) = 0$ for every $y\in \zer (V + \partial \hat{f})$. Indeed, write $0 = V(y) + \xi_{y}$ for some $\xi_{y} \in \partial \hat{f}(y)$. Then, by using the subgradient inequality, we come to 
    \[
        0 \leq F_{x^{*}}(y) = \langle y - x^{*}, V(y)\rangle + \hat{f}(y) - \hat{f}(x^{*}) \leq \langle y - x^{*}, V(y) + \xi_{y}\rangle = 0. 
    \]
\end{remark}

\subsection{Energy functions and preliminary estimates}
To study the convergence properties of Algorithm \ref{alg:fb-reg-eg}, we will be analyzing the descent properties of different energy functions. First, we will show a descent property for a distance-based energy function anchored to an optimal solution to the bilevel problem \eqref{eq:bilevel-where-inner-is-monotone-inclusion}.
\begin{lemma}\label{lem:fb-eg-iterate-descent}
Let $(x_{k})_{k\geq 0}$ and $(y_{k})_{k\geq 0}$ be the sequences given by Algorithm \ref{alg:fb-reg-eg}, and let $x^{*}$ be an optimal solution to \eqref{eq:bilevel-where-inner-is-monotone-inclusion}. Then, for every $k\geq 0$, we have 
\begin{align}
        \| x_{k + 1} - x^{*}\|^{2} &\leq \| x_{k} - x^{*}\|^{2} - (1 - s L_{k}) \Bigl[ \| x_{k} - y_{k}\|^{2} + \| x_{k + 1} - y_{k}\|^{2} \Bigr] \nonumber\\ 
        &\rph{\leq }- 2s \Bigl[ \langle y_{k} - x^{*}, \Phi_{k}(y_{k})\rangle + g_{k}(y_{k}) - g_{k}(x^{*})\Bigr]. \label{eq:fb-eg-iterate-descent}
\end{align}
\end{lemma}
\begin{proof}
 Let $k \geq 0$.  From the definition of the proximal operator, we have 
    \begin{equation}\label{eq:fb-eg-rewritten}
        \frac{x_{k} - s\Phi_{k}(x_{k}) - y_{k}}{s} \in \partial g_{k}(y_{k}) \quad \text{and} \quad \frac{x_{k} - s \Phi_{k}(y_{k}) - x_{k + 1}}{s} \in \partial g_{k}(x_{k + 1}). 
    \end{equation}
We use the subgradient inequality twice
    \begin{align*}
        g_{k}(x_{k + 1}) &\geq \left\langle \frac{x_{k} - y_{k}}{s} - \Phi_{k}(x_{k}), x_{k + 1} - y_{k}\right\rangle + g_{k}(y_{k}), \\
        g_{k}(x^{*}) &\geq \left\langle \frac{x_{k} - x_{k + 1}}{s} - \Phi_{k}(y_{k}), x^{*} - x_{k + 1}\right\rangle + g_{k}(x_{k + 1}).
    \end{align*}
Summing these inequalities and canceling the term $g_{k}(x_{k + 1})$, we obtain 
    \begin{equation*}
        g_{k}(x^{*}) \geq g_{k}(y_{k}) + \left\langle \frac{x_{k} - y_{k}}{s} - \Phi_{k}(x_{k}), x_{k + 1} - y_{k}\right\rangle + \left\langle \frac{x_{k} - x_{k + 1}}{s} - \Phi_{k}(y_{k}), x^{*} - x_{k + 1}\right\rangle,
    \end{equation*}
    which after multiplication by $s$ and rearranging gives 
    \begin{align*}
        s(g_{k}(y_{k}) - g_{k}(x^{*})) &\leq - \langle x_{k} - y_{k} - s\Phi_{k}(x_{k}), x_{k + 1} - y_{k}\rangle - \langle x_{k} - x_{k + 1} - s\Phi_{k}(y_{k}), x^{*} - x_{k + 1}\rangle \\
        &= - \langle x_{k} - y_{k}, x_{k + 1} - y_{k}\rangle - \langle x_{k} - x_{k + 1}, x^{*} - x_{k + 1}\rangle \\
        &\rph{\leq} + s \langle \Phi_{k}(x_{k}), x_{k + 1} - y_{k}\rangle + s \langle \Phi_{k}(y_{k}), x^{*} - x_{k + 1}\rangle.
    \end{align*}
We add $s \langle  y_{k} - x^{*}, \Phi_{k}(y_{k})\rangle$ to both sides of the previous inequality, which yields 
    \begin{align}
        s \Bigl[ \langle y_{k} - x^{*}, \Phi_{k}(y_{k})\rangle + g_{k}(y_{k}) - g_{k}(x^{*})\Bigr] &\leq - \langle x_{k} - y_{k}, x_{k + 1} - y_{k}\rangle - \langle x_{k} - x_{k + 1}, x^{*} - x_{k + 1}\rangle \nonumber\\
        &\rph{\leq} + s \langle y_{k} - x_{k + 1}, \Phi_{k}(y_{k}) - \Phi_{k}(x_{k})\rangle. \label{eq:fb-eg-descent-1}
    \end{align}
    By elementary properties of the inner product we have 
    \begin{align*}
        -2 \langle x_{k} - y_{k}, x_{k + 1} - y_{k}\rangle &= -\| x_{k} - y_{k}\|^{2} - \| x_{k + 1} - y_{k}\|^{2} + \| x_{k} - x_{k + 1}\|^{2}, \\
        -2 \langle x_{k} - x_{k + 1}, x^{*} - x_{k + 1}\rangle &= \| x_{k} - x^{*}\|^{2} - \| x_{k} - x_{k + 1}\|^{2} - \| x_{k + 1} - x^{*}\|^{2}. 
    \end{align*}
Now, we multiply \eqref{eq:fb-eg-descent-1} by $2$, plug the previous equalities, and use the Cauchy-Schwarz inequality, the $L_{k}$-Lipschitz continuity of $\Phi_{k}$ and Young's inequality to get  
    \begin{align*}
        &\rph{\leq} 2s \Bigl[ \langle y_{k} - x^{*}, \Phi_{k}(y_{k})\rangle + g_{k}(y_{k}) - g_{k}(x^{*})\Bigr] \\
        &\leq -\| x_{k} - y_{k}\|^{2} - \| x_{k + 1} - y_{k}\|^{2} + \| x_{k} - x_{k + 1}\|^{2} + \| x_{k} - x^{*}\|^{2} - \| x_{k} - x_{k + 1}\|^{2} - \| x_{k + 1} - x^{*}\|^{2} \\
        &\rph{\leq} + 2s \| x_{k + 1} - y_{k}\| \| \Phi_{k}(x_{k}) - \Phi_{k}(y_{k})\| \\
        &\leq -\| x_{k + 1} - x^{*}\|^{2} + \| x_{k} - x^{*}\|^{2} - \| x_{k} - y_{k}\|^{2} - \| x_{k + 1} - y_{k}\|^{2} + 2s L_{k} \|x_{k} - y_{k}\| \| x_{k + 1} - y_{k}\| \nonumber \\
        &\leq -\| x_{k + 1} - x^{*}\|^{2} + \| x_{k} - x^{*}\|^{2} - \| x_{k} - y_{k}\|^{2} - \| x_{k + 1} - y_{k}\|^{2} + s L_{k} \| x_{k} - y_{k}\|^{2} + s L_{k}\| x_{k + 1} - y_{k}\|^{2} \\
        &= - \| x_{k + 1} - x^{*}\|^{2} + \| x_{k} - x^{*}\|^{2} - (1 - s L_{k}) \Bigl[\| x_{k} - y_{k}\|^{2}  + \| x_{k + 1} - y_{k}\|^{2}\Bigr],
    \end{align*}
    so after reordering we obtain the desired statement. 
\end{proof}  
For the proximal-extragradient method, Giselsson, Latafat and Upadhyaya \cite{GiselssonLatafatUpadhyaya2026} recently introduced a new optimality measure which is formulated independently of any solution to the problem $0 \in V(x) + \partial \hat{f}(x)$. Inspired by their work, we define
\[
    \mathcal{V}_{k} := \frac{2}{s} \langle \Phi_{k}(x_{k}) - \Phi_{k}(y_{k}), x_{k} - x_{k + 1}\rangle + \frac{1}{s^{2}} \| x_{k + 1} - y_{k}\|^{2} + \frac{1}{s^{2}}\| x_{k} - x_{k + 1}\|^{2} \quad \forall k \geq 0.
\]
Notice that from \eqref{eq:fb-eg-rewritten}, we can rewrite Algorithm \ref{alg:fb-reg-eg} as 
\begin{equation}\label{eq:alg-2-rewritten}
    (\forall k \geq 0) \quad \left\{
        \begin{aligned}
            y_{k} &:= x_{k} - s\Phi_{k}(x_{k}) - s\zeta_{k}, & \ \zeta_{k} &\in \partial g_{k}(y_{k}), \\
            x_{k + 1} &:= x_{k} - s\Phi_{k}(y_{k}) - s \xi_{k + 1}, & \ \xi_{k + 1} &\in \partial g_{k}(x_{k + 1}).
        \end{aligned}
    \right.
\end{equation}
Moreover, for every $k\geq 1$, we write
\begin{equation}\label{eq:alg-2-rewritten-2}
    \xi_{k} := \xi_{\hat{f}, k} + \varepsilon_{k - 1} \xi_{\hat{h}, k}, \quad \xi_{\hat{f}, k} \in \partial \hat{f}(x_{k}), \ \xi_{\hat{h}, k} \in \partial \hat{h}(x_{k}).
\end{equation}
The following lemma shows useful upper and lower bounds for $(\mathcal{V}_{k})_{k \geq 0}$, as well as the fundamental fact that $(\mathcal{V}_{k})_{k \geq 0}$ upper bounds the norm squared of the full regularized operator. 
\begin{lemma}\label{lem:lower-upper-bound-nuk}
Let $(x_{k})_{k\geq 0}$ and $(y_{k})_{k\geq 0}$ be the sequences given by Algorithm \ref{alg:fb-reg-eg}. Then, for every sufficiently large $k \geq 0$, the following statements hold:
    \begin{enumerate}[\rm (i)]
        \item $\mathcal{V}_{k} \geq \frac{1 - s L_{k}}{s^{2}} \Bigl[ \| x_{k + 1} - y_{k}\|^{2} + \| x_{k + 1} - x_{k}\|^{2}\Bigr] \geq 0$;
        \item $ \mathcal{V}_{k} \leq \frac{5}{s^{2}} \Bigl[ \| x_{k + 1} - y_{k}\|^{2} + \| x_{k} - y_{k}\|^{2}\Bigr]$;
        \item $ \| \Phi_{k}(x_{k + 1}) + \xi_{k + 1}\|^{2} \leq \mathcal{V}_{k}$. 
    \end{enumerate}
\end{lemma}
\begin{proof}
(i) Using the $L_{k}$-Lipschitz continuity and the monotonicity of $\Phi_{k}$ together with Young's inequality, for every $k \geq 0$, we have
    \begin{align*}
        \mathcal{V}_{k} &= \frac{2}{s} \langle x_{k} - x_{k + 1}, \Phi_{k}(x_{k}) - \Phi_{k}(y_{k})\rangle + \frac{1}{s^{2}} \| x_{k + 1} - y_{k}\|^{2} + \frac{1}{s^{2}}\| x_{k} - x_{k + 1}\|^{2} \\
        &= \frac{2}{s}\langle x_{k} - x_{k + 1}, \Phi_{k}(x_{k}) - \Phi_{k}(x_{k + 1})\rangle + \frac{2}{s}\langle x_{k} - x_{k + 1}, \Phi_{k}(x_{k + 1}) - \Phi_{k}(y_{k})\rangle \\
        &\rph{=} + \frac{1}{s^{2}} \| x_{k + 1} - y_{k}\|^{2} + \frac{1}{s^{2}}\| x_{k} - x_{k + 1}\|^{2} \\
        &\geq - \frac{2}{s} \| x_{k} - x_{k + 1}\| L_{k} \| x_{k + 1} - y_{k}\| + \frac{1}{s^{2}} \| x_{k + 1} - y_{k}\|^{2} + \frac{1}{s^{2}}\| x_{k} - x_{k + 1}\|^{2} \\
        &\geq -\frac{L_{k}}{s} \| x_{k + 1} - x_{k}\|^{2} - \frac{L_{k}}{s} \| x_{k + 1} - y_{k}\|^{2} + \frac{1}{s^{2}} \| x_{k + 1} - y_{k}\|^{2} + \frac{1}{s^{2}}\| x_{k} - x_{k + 1}\|^{2} \\
        &= \frac{1 - s L_{k}}{s^{2}} \Bigl[ \| x_{k + 1} - y_{k}\|^{2} + \| x_{k + 1} - x_{k}\|^{2}\Bigr]. 
    \end{align*}

    (ii) Using the $L_{k}$-Lipschitz continuity of $\Phi_{k}$ together with Young's inequality and the fact that $L_{k} \leq \frac{1}{s}$, for every sufficiently large $k \geq 0$, we obtain 
    \begin{align*}
        \mathcal{V}_{k} &\leq \frac{2}{s} \| x_{k} - x_{k + 1}\| L_{k} \| x_{k} - y_{k}\| + \frac{1}{s^{2}} \| x_{k + 1} - y_{k}\|^{2} + \frac{2}{s^{2}} \| x_{k + 1} - y_{k}\|^{2} + \frac{2}{s^{2}} \| x_{k} - y_{k}\|^{2} \\
        &\leq \frac{1}{s^{2}} \| x_{k + 1} - x_{k}\|^{2} + \frac{1}{s^{2}} \| x_{k} - y_{k}\|^{2} + \frac{1}{s^{2}} \| x_{k + 1} - y_{k}\|^{2} + \frac{2}{s^{2}} \| x_{k + 1} - y_{k}\|^{2} + \frac{2}{s^{2}} \| x_{k} - y_{k}\|^{2} \\
        &\leq \frac{2}{s^{2}} \| x_{k + 1} - y_{k}\|^{2} + \frac{2}{s^{2}} \| x_{k} - y_{k}\|^{2} + \frac{1}{s^{2}} \| x_{k} - y_{k}\|^{2} + \frac{1}{s^{2}} \| x_{k + 1} - y_{k}\|^{2} + \frac{2}{s^{2}} \| x_{k + 1} - y_{k}\|^{2} + \frac{2}{s^{2}} \| x_{k} - y_{k}\|^{2} \\
        &= \frac{5}{s^{2}} \Bigl[ \| x_{k + 1} - y_{k}\|^{2} + \| x_{k} - y_{k}\|^{2}\Bigr].
    \end{align*}

    (iii) Using the monotonicity of $\Phi_{k}$ together the fact that $L_{k} \leq \frac{1}{s}$, for every sufficiently large $k \geq 0$, we obtain 
    \begin{align*}
        \| \Phi_{k}(x_{k + 1}) + \xi_{k + 1}\|^{2} &= \Bigl\| \frac{1}{s} (x_{k} - x_{k + 1}) - \Phi_{k}(y_{k}) + \Phi_{k}(x_{k + 1})\Bigr\|^{2} \\
        &= \| \Phi_{k}(x_{k + 1}) - \Phi_{k}(y_{k})\|^{2} + \frac{1}{s^{2}} \| x_{k} - x_{k + 1}\|^{2} + \frac{2}{s} \langle \Phi_{k}(x_{k + 1}) - \Phi_{k}(y_{k}), x_{k} - x_{k + 1}\rangle \\
        &\leq L_{k}^{2} \| x_{k + 1} - y_{k}\|^{2} + \frac{1}{s^{2}} \| x_{k} - x_{k + 1}\|^{2} + \frac{2}{s} \langle \Phi_{k}(x_{k + 1}) - \Phi_{k}(y_{k}), x_{k} - x_{k + 1}\rangle \\
        &\leq \frac{1}{s^{2}} \| x_{k + 1} - y_{k}\|^{2} + \frac{1}{s^{2}} \| x_{k} - x_{k + 1}\|^{2} + \frac{2}{s} \langle \Phi_{k}(x_{k}) - \Phi_{k}(y_{k}), x_{k} - x_{k + 1}\rangle \\
        &= \mathcal{V}_{k}.
    \end{align*}
\end{proof}
Now, we show the main descent property for $\mathcal{V}_{k}$ which will be fundamental for our analysis. 
\begin{lemma}\label{lem:descent-nuk}
    Let $(x_{k})_{k\geq 0}$ and $ (y_{k})_{k\geq 0}$ be the sequences given by Algorithm \ref{alg:fb-reg-eg}. Then, for every $k\geq 0$, it holds 
    \[
        \mathcal{V}_{k + 1} \leq \mathcal{V}_{k} - \frac{(1 - s^{2} L_{k}^{2})}{s^{2}} \| x_{k + 1} - y_{k}\|^{2} - \frac{2 (\Delta\varepsilon_{k})}{s} (H(y_{k + 1}) - H(x_{k + 1})),
    \]
where $\Delta\varepsilon_{k} := \varepsilon_{k+1} - \varepsilon_{k} \leq 0$.
\end{lemma}
\begin{proof}
According to \eqref{eq:fb-eg-rewritten}, for every $k \geq 0$, we have  
    \begin{equation*}
        \frac{x_{k} - s\Phi_{k}(x_{k}) - y_{k}}{s} \in \partial g_{k}(y_{k}) \quad \text{and} \quad \frac{x_{k} - s \Phi_{k}(y_{k}) - x_{k + 1}}{s} \in \partial g_{k}(x_{k + 1}). 
    \end{equation*}
    Using the previous statements for steps $k$ and $k + 1$, we obtain, for every $k \geq 0$ and every $x \in \cH$,  the following subgradient inequalities
    \begin{align}
        g_{k}(x) &\geq g_{k}(x_{k + 1}) + \left\langle \frac{1}{s}(x_{k} - x_{k + 1}) - \Phi_{k}(y_{k}), x - x_{k + 1}\right\rangle, \label{eq:subgradient-ineq-1}\\
        g_{k + 1}(x) &\geq g_{k + 1}(x_{k + 2}) + \left\langle \frac{1}{s} (x_{k + 1} - x_{k + 2}) - \Phi_{k + 1}(y_{k + 1}), x - x_{k + 2}\right\rangle, \label{eq:subgradient-ineq-2}\\
        g_{k + 1}(x) &\geq g_{k + 1}(y_{k + 1}) + \left\langle \frac{1}{s}(x_{k + 1} - y_{k + 1}) - \Phi_{k + 1}(x_{k + 1}), x - y_{k + 1}\right\rangle. \label{eq:subgradient-ineq-3}
    \end{align}
    We plug $x := y_{k + 1}$ in \eqref{eq:subgradient-ineq-1}, $x := x_{k + 1}$ in \eqref{eq:subgradient-ineq-2} and $x := x_{k + 2}$ in \eqref{eq:subgradient-ineq-3}. We make $g_{k + 1}$ appear in the first inequality and then we sum all three together. This yields, for every $k \geq 0$,
    \begin{align}
        0 &\geq \Bigl[ g_{k}(x_{k + 1}) - g_{k + 1}(x_{k + 1}) - g_{k}(y_{k + 1}) + g_{k + 1}(y_{k + 1})\Bigr] \nonumber\\
        &\rph{\geq} + (g_{k + 1}(x_{k + 1}) - g_{k + 1}(y_{k + 1})) + (g_{k + 1}(x_{k + 2}) - g_{k + 1}(x_{k + 1})) + (g_{k + 1}(y_{k + 1}) - g_{k + 1}(x_{k + 2})) \nonumber \\
        &\rph{\geq} + \left\langle \frac{1}{s}(x_{k} - x_{k + 1}) - \Phi_{k}(y_{k}), y_{k + 1} - x_{k + 1}\right\rangle + \left\langle \frac{1}{s} (x_{k + 1} - x_{k + 2}) - \Phi_{k + 1}(y_{k + 1}), x_{k + 1} - x_{k + 2}\right\rangle \nonumber\\
        &\rph{\geq} + \left\langle \frac{1}{s}(x_{k + 1} - y_{k + 1}) - \Phi_{k + 1}(x_{k + 1}), x_{k + 2} - y_{k + 1}\right\rangle \nonumber \\
        &= - (\Delta \varepsilon_{k}) (\hat{h}(x_{k + 1}) - \hat{h}(y_{k + 1})) \nonumber\\
        &\rph{=} + \frac{1}{s} \langle x_{k} - x_{k + 1}, y_{k + 1} - x_{k + 1}\rangle - \langle \Phi_{k}(y_{k}), y_{k + 1} - x_{k + 1} \rangle \label{eq:descent-nu-1}\\
        &\rph{=} + \frac{1}{s} \| x_{k + 1} - x_{k + 2}\|^{2} - \langle \Phi_{k + 1}(y_{k + 1}), x_{k + 1} - x_{k + 2}\rangle \label{eq:descent-nu-2}\\
        &\rph{=} + \frac{1}{s} \langle x_{k + 1} - y_{k + 1}, x_{k + 2} - y_{k + 1}\rangle - \langle \Phi_{k + 1}(x_{k + 1}), x_{k + 2} - y_{k + 1}\rangle. \label{eq:descent-nu-3}
    \end{align}    
We multiply everything by $\frac{2}{s}$. Call $A_{k}$ to the resulting line \eqref{eq:descent-nu-1}, and call $B_{k}$ to the resulting lines \eqref{eq:descent-nu-2} and \eqref{eq:descent-nu-3}. For every $k \geq 0$, $A_{k}$ can be simplified as 
    \[
        A_{k} = \frac{1}{s^{2}} \| x_{k} - x_{k + 1}\|^{2} + \frac{1}{s^{2}} \| y_{k + 1} - x_{k + 1}\|^{2} - \frac{1}{s^{2}} \| x_{k} - y_{k + 1}\|^{2} - \frac{2}{s} \langle \Phi_{k}(y_{k}), y_{k + 1} - x_{k + 1}\rangle, 
    \]
while the terms in $B_{k}$ read
    \begin{align*}
        B_{k} &= \frac{2}{s^{2}} \| x_{k + 1} - x_{k + 2}\|^{2} + \frac{2}{s} \langle \Phi_{k + 1}(x_{k + 1}) - \Phi_{k + 1}(y_{k + 1}), x_{k + 1} - x_{k + 2}\rangle - \frac{2}{s} \langle \Phi_{k + 1}(x_{k + 1}), x_{k + 1} - x_{k + 2}\rangle \\
        &\rph{=} + \frac{1}{s^{2}} \| x_{k + 1} - y_{k + 1}\|^{2} + \frac{1}{s^{2}} \| x_{k + 2} - y_{k + 1}\|^{2} - \frac{1}{s^{2}} \| x_{k + 1} - x_{k + 2}\|^{2} - \frac{2}{s} \langle \Phi_{k + 1}(x_{k + 1}), x_{k + 2} - y_{k + 1}\rangle \\
        &= \frac{2}{s} \langle \Phi_{k + 1}(x_{k + 1}) - \Phi_{k + 1}(y_{k + 1}), x_{k + 1} - x_{k + 2}\rangle + \frac{1}{s^{2}} \| x_{k + 2} - y_{k + 1}\|^{2} + \frac{1}{s^{2}} \| x_{k + 1} - x_{k + 2}\|^{2} \\
        &\rph{=} + \frac{1}{s^{2}}\| x_{k + 1} - y_{k + 1}\|^{2} - \frac{2}{s} \langle \Phi_{k + 1}(x_{k + 1}), x_{k + 1} - y_{k + 1}\rangle \\
        &= \mathcal{V}_{k + 1} + \frac{1}{s^{2}} \| x_{k + 1} - y_{k + 1}\|^{2} - \frac{2}{s} \langle \Phi_{k + 1}(x_{k + 1}), x_{k + 1} - y_{k + 1}\rangle.
    \end{align*}
For every $k \geq 0$, we have $0 \geq - \frac{2(\Delta \varepsilon_{k})}{s} (\hat{h}(x_{k + 1}) - \hat{h}(y_{k + 1})) + A_{k} + B_{k}$ or, equivalently,
    \begin{align}
        \mathcal{V}_{k + 1} &\leq - \frac{2(\Delta\varepsilon_{k})}{s} (\hat{h}(y_{k + 1}) - \hat{h}(x_{k + 1})) \nonumber\\
        &\rph{\leq} - \frac{1}{s^{2}} \| x_{k + 1} - y_{k + 1}\|^{2} + \frac{2}{s} \langle \Phi_{k + 1}(x_{k + 1}), x_{k + 1} - y_{k + 1}\rangle \nonumber\\
        &\rph{\leq} - \frac{1}{s^{2}} \| x_{k} - x_{k + 1}\|^{2} - \frac{1}{s^{2}} \| y_{k + 1} - x_{k + 1}\|^{2} + \frac{1}{s^{2}} \| x_{k} - y_{k + 1}\|^{2} + \frac{2}{s} \langle \Phi_{k}(y_{k}), y_{k + 1} - x_{k + 1}\rangle \nonumber\\
        &= - \frac{2(\Delta\varepsilon_{k})}{s} (\hat{h}(y_{k + 1}) - \hat{h}(x_{k + 1})) \nonumber\\
        &\rph{=} + \frac{1}{s^{2}} \| x_{k} - y_{k + 1}\|^{2} - \frac{1}{s^{2}} \| x_{k} - x_{k + 1}\|^{2} - \frac{2}{s^{2}} \| x_{k + 1} - y_{k + 1}\|^{2} + \frac{2}{s} \langle \Phi_{k}(y_{k}) - \Phi_{k}(x_{k + 1}), y_{k + 1} - x_{k + 1}\rangle \nonumber\\
        &\rph{=} + \frac{2}{s} \langle \Phi_{k}(x_{k + 1}) - \Phi_{k + 1}(x_{k + 1}), y_{k + 1} - x_{k + 1}\rangle \nonumber\\
        &= - \frac{2(\Delta\varepsilon_{k})}{s} \Bigl[ \hat{h}(y_{k + 1}) - \hat{h}(x_{k + 1}) + \langle \nabla h(x_{k + 1}), y_{k + 1} - x_{k + 1}\rangle\Bigr] \nonumber\\
        &\rph{=} + \frac{1}{s^{2}} \| x_{k} - y_{k + 1}\|^{2} - \frac{1}{s^{2}} \| x_{k} - x_{k + 1}\|^{2} - \frac{2}{s^{2}} \| x_{k + 1} - y_{k + 1}\|^{2} \nonumber\\
        &\rph{=} + \frac{2}{s} \langle \Phi_{k}(y_{k}) - \Phi_{k}(x_{k + 1}), y_{k + 1} - x_{k + 1}\rangle. \label{eq:descent-nu-4}
    \end{align}
    Notice that, for every $k \geq 0$, 
    \begin{align*}
        \langle \Phi_{k}(y_{k}) - \Phi_{k}(x_{k + 1}), y_{k + 1} - x_{k + 1}\rangle &= \langle \Phi_{k}(x_{k}) - \Phi_{k}(x_{k + 1}), x_{k + 1} - x_{k}\rangle + \langle \Phi_{k}(y_{k}) - \Phi_{k}(x_{k}), x_{k + 1} - x_{k}\rangle \\
        &\rph{=}+ \bigl\langle \Phi_{k}(y_{k}) - \Phi_{k}(x_{k + 1}), y_{k + 1} + x_{k} - 2x_{k + 1}\bigr\rangle \\
        &\leq \langle \Phi_{k}(y_{k}) - \Phi_{k}(x_{k}), x_{k + 1} - x_{k}\rangle + \frac{s}{2} \| \Phi_{k}(y_{k}) - \Phi_{k}(x_{k + 1})\|^{2} \\
        &\rph{\leq} + \frac{1}{2s} \bigl\| y_{k + 1} + x_{k} - 2x_{k + 1}\bigr\|^{2} \\
        &\leq \langle \Phi_{k}(y_{k}) - \Phi_{k}(x_{k}), x_{k + 1} - x_{k}\rangle + \frac{s L_{k}^{2}}{2} \| y_{k} - x_{k + 1}\|^{2} \\
        &\rph{\leq} + \frac{1}{2s} \Bigl[ 2 \| x_{k} - x_{k + 1}\|^{2} + 2\| y_{k + 1} - x_{k + 1}\|^{2} - \| x_{k} - y_{k + 1}\|^{2}\Bigr], 
    \end{align*}
    where we used the $L_{k}$-Lipschitz continuity and monotonicity of $\Phi_{k}$, Young's inequality, and the polarization identity. Multiplying the previous inequality by $\frac{2}{s}$ and plugging it into \eqref{eq:descent-nu-4} gives us, after grouping the corresponding terms together, for every $k \geq 0$,
    \begin{align*}
        \mathcal{V}_{k + 1} &\leq - \frac{2(\Delta\varepsilon_{k})}{s} \Bigl[ \hat{h}(y_{k + 1}) - \hat{h}(x_{k + 1}) + \langle \nabla h(x_{k + 1}), y_{k + 1} - x_{k + 1}\rangle\Bigr] \\
        &\rph{\leq} + \frac{2}{s}\langle \Phi_{k}(y_{k}) - \Phi_{k}(x_{k}), x_{k + 1} - x_{k}\rangle + \frac{1}{s^{2}} \| x_{k} - x_{k + 1}\|^{2} + L_{k}^{2}\| x_{k + 1} - y_{k}\|^{2} \\
        &\leq - \frac{2(\Delta\varepsilon_{k})}{s} \Bigl[ \hat{h}(y_{k + 1}) - \hat{h}(x_{k + 1}) + h(y_{k + 1}) - h(x_{k + 1})\Bigr] \\
        &\rph{\leq} + \frac{2}{s}\langle \Phi_{k}(y_{k}) - \Phi_{k}(x_{k}), x_{k + 1} - x_{k}\rangle + \frac{1}{s^{2}} \| x_{k + 1} - y_{k}\|^{2} + \frac{1}{s^{2}} \| x_{k} - x_{k + 1}\|^{2} \\
        &\rph{\leq} - \frac{(1 - s^{2} L_{k}^{2})}{s^{2}} \| x_{k + 1} - y_{k}\|^{2} \\
        &= - \frac{2(\Delta \varepsilon_{k})}{s} (H(y_{k + 1}) - H(x_{k + 1})) + \mathcal{V}_{k} - \frac{(1 - s^{2} L_{k}^{2})}{s^{2}} \| x_{k + 1} - y_{k}\|^{2}, 
    \end{align*}
    where we used the gradient inequality on $h$ together with the fact that $-\Delta \varepsilon_{k} \geq 0$. 
\end{proof}
Now, we define the energy function 
\[
    E_{k} := (k - 1) \mathcal{V}_{k} + \frac{2k (\Delta \varepsilon_{k})}{s} H(y_{k}) \quad \forall k \geq 1,
\]
and state its descent property in the following lemma. 
\begin{lemma}\label{lem:descent-E-discrete}
    Let $(x_{k})_{k\geq 0}$ and $(y_{k})_{k\geq 0}$ be the sequences given by Algorithm \ref{alg:fb-reg-eg}. Then, for every sufficiently large  $k \geq 1$, it holds
    \[
        \Delta E_{k} \leq \mathcal{V}_{k} + \frac{2 k (\Delta\varepsilon_{k})}{s} (H(x_{k + 1}) - H(y_{k})) - \frac{k (1 - s^{2} L_{k}^{2})}{s^{2}} \| x_{k + 1} - y_{k}\|^{2} + \frac{2\Delta( k (\Delta \varepsilon_{k}))}{s} H(y_{k + 1}),
    \]
where $\Delta E_{k} := E_{k+1}-E_k$ and $\Delta( k (\Delta \varepsilon_{k})):=(k+1)(\Delta \varepsilon_{k+1}) - k (\Delta \varepsilon_{k})$.
\end{lemma}
\begin{proof}
    According to Lemma \ref{lem:descent-nuk}, for every $k \geq 1$, we have 
    \begin{align*}
        \Delta E_{k} &= \mathcal{V}_{k} + k (\mathcal{V}_{k+1} - \mathcal{V}_{k}) + \frac{2\Delta( k (\Delta \varepsilon_{k}))}{s} H(y_{k + 1}) + \frac{2 k (\Delta \varepsilon_{k})}{s} (H(y_{k + 1}) - H(y_{k})) \\
        &\leq \mathcal{V}_{k} + \left[ - \frac{2 k(\Delta \varepsilon_{k})}{s} (H(y_{k + 1}) - H(x_{k + 1})) - \frac{k (1 - s^{2} L_{k}^{2})}{s^{2}} \| x_{k + 1} - y_{k}\|^{2}\right] \\
        &\rph{\leq} + \frac{2\Delta( k (\Delta \varepsilon_{k}))}{s} H(y_{k + 1}) + \frac{2 k (\Delta \varepsilon_{k})}{s} (H(y_{k + 1}) - H(y_{k})) \\
        &= \mathcal{V}_{k} + \frac{2 k (\Delta\varepsilon_{k})}{s} (H(x_{k + 1}) - H(y_{k})) - \frac{k (1 - s^{2} L_{k}^{2})}{s^{2}} \| x_{k + 1} - y_{k}\|^{2} + \frac{2\Delta( k (\Delta \varepsilon_{k}))}{s} H(y_{k + 1}).
    \end{align*}
\end{proof}
Similar to the continuous-time case, we will now show some partial convergence guarantees which hold in the general setting of Assumption \ref{ass:standing-disc}, with the added assumption that $(x_{k})_{k\geq 0}$ is bounded. First, for $\lambda > 0$ and $x^{*}$ an optimal solution to \eqref{eq:bilevel-where-inner-is-monotone-inclusion}, we define the energy 
\[
    E_{\lambda, k} := E_{k} + \lambda \| x_{k} - x^{*}\|^{2} \quad \forall k \geq 1,
\]  
and state the relevant descent property in the following lemma. 
\begin{lemma}\label{lem:ergodic-rate-ineq-disc}
    Let $(x_{k})_{k\geq 0}$ and $(y_{k})_{k\geq 0}$ be the sequences given by Algorithm \ref{alg:fb-reg-eg}, and let $x^{*}$ be an optimal solution to \eqref{eq:bilevel-where-inner-is-monotone-inclusion}. Assume that $(x_{k})_{k\geq 0}$ is bounded.  Then, for every sufficiently large  $k \geq 1$, it holds 
    \begin{align*}
        &\rph{\leq}\Delta E_{\lambda, k} + \left(\frac{\lambda s^{2}(1 - sL_{k})}{5} - 1\right) \mathcal{V}_{k} + 2s\lambda \varepsilon_{k} (H(y_{k}) - H(x^{*})) \\
        &\leq -2 s \lambda F_{x^{*}}(y_{k}) + \frac{C_{H}}{\mu s} k |\Delta \varepsilon_{k}|^{2} + \frac{2\Delta (k(\Delta \varepsilon_{k}))}{s} H(y_{k + 1}), 
    \end{align*}
    for certain constants $C_{H} >0$ and $\mu > 0$, where $\Delta E_{\lambda, k}:= E_{\lambda, k+1} - E_{\lambda, k}.$
\end{lemma}
\begin{proof}
We use the local Lipschitz continuity of $H$ on the bounded sequences $(x_{k})_{k\geq 0}$ and $(y_{k})_{k\geq 0}$, which is a consequence of $\nabla h$ being Lipschitz continuous and $\partial \hat{h}$ mapping bounded subsets of its domain to bounded subsets of $\cH$. For some constant $C_{H} \geq 0$, according to Young's inequality, for every $\mu > 0$ and every $k \geq 1$, we have 
    \[
        \frac{2 k | \Delta \varepsilon_{k}|}{s} | H(x_{k + 1}) - H(y_{k})| \leq \frac{2 C_{H} k | \Delta\varepsilon_{k}|}{s} \| x_{k + 1} - y_{k}\| \leq \frac{\mu C_{H} s}{s^{2}} k \|x_{k + 1} - y_{k}\|^{2} + \frac{C_{H}}{\mu s } k | \Delta \varepsilon_{k}|^{2}. 
    \]
    Choose $\mu > 0$ small enough such that $0 < \mu C_{H} s < 1 - s^{2} L_{V}^{2}$. Plugging the previous line into the descent property given by Lemma \ref{lem:descent-E-discrete} gives, for every $k \geq 1$, 
    \begin{equation}\label{eq:descent-Ek-2}
        \Delta E_{k} \leq \mathcal{V}_{k} - \frac{k (1 - s^{2}L_{k}^{2} - \mu C_{H}s)}{s^{2}} \| x_{k + 1} - y_{k}\|^{2} + \frac{C_{H}}{\mu s} k | \Delta \varepsilon_{k}|^{2} + \frac{2 \Delta (k (\Delta\varepsilon_{k}))}{s} H(y_{k + 1}).
    \end{equation}
    Now, recall part (ii) of Lemma \ref{lem:lower-upper-bound-nuk}. We combine it with the descent property provided by Lemma \ref{lem:fb-eg-iterate-descent} together with the gradient inequality applied to $h$ to produce, for every sufficiently large  $k \geq 1$,
    \begin{align*}
        \| x_{k + 1} - x^{*}\|^{2} &\leq \| x_{k} - x^{*}\|^{2} - \frac{s^{2} (1 - s L_{k})}{5} \mathcal{V}_{k} - 2s \Bigl[ \langle y_{k} - x^{*}, \Phi_{k}(y_{k})\rangle + g_{k}(y_{k}) - g_{k}(x^{*})\Bigr] \\
        &\leq \| x_{k} - x^{*}\|^{2} - \frac{s^{2} (1 - sL_{k})}{5} \mathcal{V}_{k} - 2s F_{x^{*}}(y_{k}) - 2s \varepsilon_{k} (H(y_{k}) - H(x^{*})).
    \end{align*}
    We multiply the previous inequality by $\lambda$ and add it to \eqref{eq:descent-Ek-2}, and we drop the norm squared term corresponding to $\| x_{k + 1} - y_{k}\|^{2}$, which yields, for every sufficiently large  $k \geq 1$, 
    \begin{align*}
        \Delta E_{\lambda, k} &\leq \left(1 - \frac{\lambda s^{2} (1 - s L_{k})}{5}\right) \mathcal{V}_{k} - 2\lambda s F_{x^{*}}(y_{k}) - 2\lambda s \varepsilon_{k} (H(y_{k}) - H(x^{*})) + \frac{C_{H}}{\mu s} k |\Delta \varepsilon_{k}|^{2} \\
        &\rph{\leq} + \frac{2 \Delta (k (\Delta \varepsilon_{k}))}{s} H(y_{k + 1}).
    \end{align*}
This is the desired statement.
\end{proof}
Now, we choose $\lambda > 0$ large enough such that $\frac{\lambda s^{2} (1 - sL_{k})}{5} > 1$ for every sufficiently large $k \geq 1$. The following theorem provides the partial convergence guarantees which hold under the standing Assumption \ref{ass:standing-disc} and provided $(x_{k})_{k\geq 0}$ is bounded.
\begin{theorem}\label{thm:ergodic-rate-discrete}
    Let $(x_{k})_{k\geq 0}$ and $(y_{k})_{k\geq 0}$ be the sequences given by Algorithm \ref{alg:fb-reg-eg}, and let $x^{*}$ be an optimal solution to \eqref{eq:bilevel-where-inner-is-monotone-inclusion}. Assume that $(x_{k})_{k\geq 0}$ is bounded. Define, for every $k \geq 1$, 
    \[
        \bar{y}_{k} := \frac{1}{\sum_{\ell = 1}^{k} \zeta_{\ell}} \sum_{\ell = 1}^{k} \zeta_{\ell} y_{\ell}, 
    \]
    where $\zeta_{k} := 2\lambda s \varepsilon_{k}$. Let $(\xi_{\hat{f}, k})_{k\geq 1}$ be given by \eqref{eq:alg-2-rewritten-2}. Then, as $k\to +\infty$, we have the convergence rates as shown by the following table:
    \begin{whitetablebox}
        \begin{tblr}{
            width = \linewidth,
            colspec = {
                Q[c,m,wd=0.18\linewidth]
                X[c,m]
                X[c,m]
                X[c,m]
            },
            cells = {
                mode = math,
            },
            cell{1}{2-Z} = {
                cmd = \textstyle,
            },
            cell{2-Z}{2-Z} = {
                bg = gray!10,
                cmd = \displaystyle,
            },
            rows = {
                rowsep = 8pt,
            },
            columns = {
                colsep = 8pt,
            },
            cell{2-Z}{2-Z} = {
                bg = gray!10,
            },
            hline{3,4} = {2-Z}{1.2pt,white},
            vline{3,4} = {2-Z}{1.2pt,white},
        }
            \toprule
            &
            \| V(x_{k}) + \xi_{\hat{f}, k}\|
            &
            \frac{1}{k} \sum_{\ell = 1}^{k} F_{x^{*}}(y_{\ell})
            &
            H(\bar{y}_{k}) - H(x^{*})
            \\
            \midrule
            \delta>1
            &
            \mathcal{O} \left(\frac{1}{\sqrt{k}}\right)
            &
            \mathcal{O} \left(\frac{1}{k}\right)
            &
            \text{--}
            \\
            \delta=1
            &
            \mathcal{O} \left(\sqrt{\frac{\ln k}{k}}\right)
            &
            \mathcal{O} \left(\frac{\ln k}{k}\right)
            &
            \leq
            \mathcal{O} \left(\frac{1}{\ln k}\right)
            \\
            0 < \delta < 1 
            &
            \mathcal{O}\left(\frac{1}{k^{\frac{\delta}{2}}}\right)
            &
            \mathcal{O}\left(\frac{1}{k^{\delta}}\right)
            &
            \leq \mathcal{O}\left(\frac{1}{k^{1 - \delta}}\right)
            \\
            \bottomrule
        \end{tblr}
    \end{whitetablebox}
    where we write $\leq$ to emphasize that we only obtain an upper bound for the sign-less quantity $H(\bar{y}_{k}) - H(x^{*})$. 
\end{theorem}
\begin{proof}
    The proof proceeds in a virtually identical fashion as the proof for Theorem \ref{thm:ergodic-rate-cont}, so we don't include the full details. We choose $k \geq 1$, sum the inequality given by Lemma \ref{lem:ergodic-rate-ineq-disc} from $\ell = 1$ to $k$ (without loss of generality, all ``for large enough $k$'' statements can be assumed to hold starting at $k = 1$) and apply Jensen's inequality to $H$ to obtain 
    \begin{align*}
        &\rph{\leq} k \mathcal{V}_{k + 1} + \frac{2(k + 1) (\Delta \varepsilon_{k + 1})}{s} H(y_{k + 1}) + \left(\sum_{\ell = 1}^{k} \zeta_{\ell}\right) (\inf H - H(x^{*})) \\
        &\leq k \mathcal{V}_{k + 1} + \frac{2(k + 1) (\Delta \varepsilon_{k + 1})}{s} H(y_{k + 1}) + \left(\sum_{\ell = 1}^{k} \zeta_{\ell}\right) (H(\bar{y}_{k}) - H(x^{*})) \\
        &\leq k \mathcal{V}_{k + 1} + \frac{2(k + 1) (\Delta \varepsilon_{k + 1})}{s} H(y_{k + 1}) + \sum_{\ell = 1}^{k} \zeta_{\ell} (H(y_{\ell}) - H(x^{*})) \\
        &\leq E_{\lambda, k + 1} + \sum_{\ell = 1}^{k} \zeta_{\ell} (H(y_{\ell}) - H(x^{*})) \\
        &\leq -2s \sum_{\ell = 1}^{k}  F_{x^{*}}(y_{\ell}) + \sum_{\ell = 1}^{k} \left[ \frac{C_{H}}{\mu s} k |\Delta \varepsilon_{k}|^{2} + \frac{2\Delta ( k(\Delta \varepsilon_{k}))}{s} H(y_{k})\right] + E_{\lambda, 1}.
    \end{align*}
    Now, just argue like in the proof for Theorem \ref{thm:ergodic-rate-cont}, taking into account that the the summands between square brackets are summable, since the boundedness of $(x_{k})_{k\geq 0}$ implies the boundedness of $(y_{k})_{k\geq 0}$ and 
    \[
        (k |\Delta \varepsilon_{k}|^{2})_{k \geq 1} = \mathcal{O} \left( \frac{1}{k^{2\delta + 1}}\right) \quad \text{and} \quad (|\Delta (k(\Delta \varepsilon_{k}))|)_{k \geq 1} = \mathcal{O} \left( \frac{1}{k^{\delta + 1}}\right)
    \]
    are summable. The term $\frac{2(k + 1) (\Delta \varepsilon_{k + 1})}{s} H(y_{k + 1})$ approaches zero as $k \to +\infty$ and thus doesn't affect the rates. In a first step, we get rates for $(\mathcal{V}_{k})_{k \geq 1}$ and $(\frac{1}{k} \sum_{\ell = 1}^{k} F_{x^{*}}(y_{\ell}))_{k \geq 1}$, and then we recall part (iii) of Lemma \ref{lem:lower-upper-bound-nuk}
    \[
        \| V(x_{k + 1}) + \xi_{\hat{f}, k + 1} + \varepsilon_{k} (\nabla h(x_{k + 1}) + \xi_{\hat{h}, k + 1})\|^{2} =\| \Phi_{ k}(x_{k + 1}) + \xi_{k + 1}\|^{2} \leq \mathcal{V}_{k} \quad \forall k \geq 0.
    \] 
    Since $(\nabla h(x_{k}))_{k\geq 0}, (\xi_{\hat{h}, k})_{k \geq 1}$ are bounded, this yields the rates for $\| V(x_{k}) + \xi_{\hat{f}, k}\|$. The upper bounds for $H(\bar{y}_{k}) - H(x^{*})$ are straightforward to check. 
\end{proof}

\begin{remark}
    In the composite smooth+nonsmooth setting, we had to use the energy function $\mathcal{V}_{k}$, which upper bounds the norm squared of the full operator. However, it is worth noting that in the purely smooth case (i.e., $\hat{f} = \hat{h} \equiv 0$), the analysis can be done using more standard measures of optimality (such as analyzing the variation of $\| \Phi_{k}(x_{k})\|^{2}$), and is closer to the continuous-time analysis. Indeed, one can show, for every $k \geq 0$ that 
    \[
        \| x_{k + 1} - x^{*}\|^{2} \leq \| x_{k} - x^{*}\|^{2} - s^{2} (1 - s^{2} L_{k}^{2}) \| \Phi_{k}(x_{k})\|^{2} - 2s \langle y_{k} - x^{*}, \Phi_{k}(y_{k})\rangle, 
    \]
    and that, if $E_{k} = \frac{s^{2}(k - 1)}{2} \| \Phi_{k}(x_{k})\|^{2} + s(\Delta \varepsilon_{k}) (h(x_{k}) - h(x^{*}))$, for every $k \geq 1$, then 
    \begin{align*}
        \Delta E_{k} &\leq - s k \langle \Phi_{k + 1}(x_{k + 1}) - \Phi_{k + 1}(x_{k}), x_{k+1} - x_k\rangle - \frac{(1 - (1 + \gamma) s^{2} L_{k}^{2}) k}{2} \| x_{k + 1} - y_{k}\|^{2} \\
        &\rph{\leq} + \frac{s^{2}}{2} \| \Phi_{k}(x_{k})\|^{2} + \frac{s^{2}}{2}\left(1 + \frac{1}{\gamma}\right) k (\Delta \varepsilon_{k})^{2} \| \nabla h(x_{k + 1})\|^{2} + s \Delta ((\Delta \varepsilon_{k}) k) (h(x_{k + 1}) - h(x^{*})),
    \end{align*}
    where $\gamma > 0$ is chosen small enough.
\end{remark}

\subsection{Geometric assumption on \texorpdfstring{$V + \partial \hat{f}$}{}}
Unfortunately, in the composite setting there is not an obvious way to state the Attouch--Czarnecki assumption in terms of the Fitzpatrick function of $V + \partial \hat{f}$. Instead, we have to settle for a summability condition expressed in terms of $F_{x^{*}}^{*}$. 
\begin{ass}\label{ass:Attouch-Czarnecki-inclusion}
    We say that the operator $V + \partial \hat{f}$ satisfies a \emph{discrete Attouch--Czarnecki-type} condition if for every optimal solution $x^{*}$ to \eqref{eq:bilevel-where-inner-is-monotone-inclusion} and every $p\in N_{\zer(V + \partial \hat{f})}(x^{*})$, it holds
    \[
        \sum_{k = 0}^{+\infty} \Bigl[ F_{x^{*}}^{*}(\varepsilon_{k} p) - \sigma_{\zer(V + \partial \hat{f})}(\varepsilon_{k} p)\Bigr] < +\infty.
    \]
\end{ass}
\noindent Similar to Lemma \ref{lem:assumptions}, the following lemma collects some important results related to Assumption \ref{ass:Attouch-Czarnecki-inclusion}.
\begin{lemma}\label{lem:assumptions-inclusion}
    Let $x^{*}$ be a solution to \eqref{eq:bilevel-where-inner-is-monotone-inclusion} and let $p\in N_{\zer(V + \partial \hat{f})}(x^{*})$, and define $\phi_{x^{*}, p} : [0, +\infty) \to \bar{\R}$ by 
    \[
        \phi_{x^{*}, p}(s) := F_{x^{*}}^{*}(s p) - \sigma_{\zer (V + \partial \hat{f})}(s p).
    \]
    Then, the following hold: 
    \begin{enumerate}[\rm (i)]
        \item $\phi_{x^{*}, p}(s) \geq 0$ for all $s\geq 0$ and $\phi_{x^{*}, p}(0) = 0$;
        \item $\phi_{x^{*}, p}$ is convex;
        \item $\phi_{x^{*}, p}$ is monotonically increasing on $[0, +\infty)$;
        \item If $p^{*} \in - \partial H(x^{*}) \cap N_{\zer (V + \partial \hat{f})}(x^{*})$, then for every $x\in \cH$ and any $s\geq 0$ we have 
        \[
            - F_{x^{*}}(x) - s (H(x) - H(x^{*})) \leq \phi_{x^{*}, p^{*}}(s)\text{;}
        \]
        \item If $V + \partial \hat{f}$ satisfies Assumption \ref{ass:Attouch-Czarnecki-inclusion}, then $\phi_{x^{*}, p}(\varepsilon_{k}) = o\left(\frac{1}{k}\right)$ as $k\to +\infty$. 
    \end{enumerate}
\end{lemma}
\begin{proof}
    (i) We recall that $F_{x^{*}}(y) = 0$ for $y \in \zer (V + \partial \hat{f})$ (see Remark \ref{rem:properties-F}) and write, for every $u \in \cH$,
    \[
        F_{x^{*}}^{*}(u) = \sup_{y\in \cH} \{ \langle y, u\rangle - F_{x^{*}}(u)\} \geq \sup_{y\in \zer(V + \partial \hat{f})} \langle y, u\rangle = \sigma_{\zer(V + \partial \hat{f})}(u), 
    \]
    which shows the nonnegativity claim. For the second claim, again recalling that $\sigma_{\zer(V + \partial \hat{f})}(sp) = \langle x^{*}, sp\rangle$, we have 
    \[
        \phi_{x^{*}, p}(0) = \sup_{y\in \cH} - F_{x^{*}} (y) = 0, 
    \]
    since $F_{x^{*}}(y) \geq 0$ for every $y\in \cH$ and the inequality is saturated when $y = x^{*}$, as per Remark \ref{rem:properties-F}. 
     
    (ii) We have 
    \[
        \phi_{x^{*}, p}(s) = \sup_{y\in \cH} \{ \langle y, sp\rangle - F_{x^{*}}(y)\} - \langle x^{*}, sp\rangle = \sup_{y\in \cH} \{ s\langle y - x^{*}, p\rangle - F_{x^{*}}(y)\},
    \]
    i.e., $\phi_{x^{*}, p}$ is a pointwise supremum of affine functions in $s$, thus convex. 

    (iii) This claim is shown exactly as in the proof of Lemma \ref{lem:assumptions}. 

    (iv) Again, we use the fact that $\sigma_{\zer (V + \partial \hat{f})}(sp) = \langle x^{*}, sp\rangle$, together with the subgradient inequality applied to $H$ to obtain, for every $x\in \cH$, 
    \begin{align*}
        - F_{x^{*}}(x) - s (H(x) - H(x^{*})) &\leq - F_{x^{*}}(x) - s\langle x - x^{*}, -p^{*}\rangle \\
        &= \langle x, sp^{*}\rangle - F_{x^{*}}(y) - \langle x^{*}, sp^{*}\rangle \\
        &\leq F_{x^{*}}^{*} (sp^{*}) - \sigma_{\zer(V + \partial \hat{f})}(sp^{*}).  
    \end{align*}

    (v) This claim is shown exactly as in the proof of Lemma \ref{lem:assumptions}. 
\end{proof}

\begin{remark}
    Similar to Remark \ref{rem:sharpness-cont}, we can also formulate a sharpness-type condition that implies Assumption \ref{ass:Attouch-Czarnecki-inclusion} in the smooth+nonsmooth setting. Namely, we say $V + \partial \hat{f}$ satisfies a sharpness assumption with exponent $\rho \in (1, 2)$ if for every optimal solution $x^{*}$ to \eqref{eq:bilevel-where-inner-is-monotone-inclusion} there exists some constant $\tau > 0$ such that for every $x\in \cH$
    \begin{equation}\label{eq:sharpness-discrete}
        \frac{\tau}{\rho} \dist (x, \zer (V + \partial \hat{f}))^{\rho} \leq \langle x - x^{*}, V(x)\rangle + \hat{f}(x) - \hat{f}(x^{*}).
    \end{equation}
    Arguing just as in Remark \ref{rem:sharpness-cont}, we can show that \eqref{eq:sharpness-discrete} implies Assumption \ref{ass:Attouch-Czarnecki-inclusion} whenever $\delta > \frac{1}{\rho^{*}}$.
\end{remark}

\subsection{Convergence analysis under Assumption \ref{ass:Attouch-Czarnecki-inclusion}}
Just as in the continuous-time setting, we will be able to derive stronger convergence guarantees under the previously introduced geometric assumptions on $V + \partial \hat{f}$. We begin with a summability result.

\begin{lemma}\label{lem:lim-xk-exists-fb-eg}
    Let $(x_{k})_{k\geq 0}$ and $(y_{k})_{k\geq 0}$ be given by Algorithm \ref{alg:fb-reg-eg}, and let $x^{*}$ be an optimal solution to \eqref{eq:bilevel-where-inner-is-monotone-inclusion}. Assume $V + \partial \hat{f}$ satisfies Assumption \ref{ass:Attouch-Czarnecki-inclusion}. Then, the limit $\lim_{k\to +\infty} \| x_{k} - x^{*}\|^{2}$ exists. Moreover, it holds 
    \[
        \sum_{k = 0}^{+\infty} \Bigl[ \| x_{k + 1} - y_{k}\|^{2} + \| x_{k} - y_{k}\|^{2}\Bigr] < +\infty.
    \]
    In particular, $(\mathcal{V}_{k})_{k\geq 0}$ is summable.
\end{lemma}
\begin{proof}
    According to Lemma \ref{lem:assumptions-inclusion}, for every $k\geq 0$, after applying the gradient inequality to $h$ and recalling the definition of $F_{x^{*}}$ we come to
    \begin{align*}
        &\rph{=} - \langle y_{k} - x^{*}, \Phi_{k}(y_{k})\rangle - (g_{k}(y_{k}) - g_{k}(x^{*})) \\
        &= - \langle y_{k} - x^{*}, V(y_{k})\rangle - (\hat{f}(y_{k}) - \hat{f}(x^{*})) - \varepsilon_{k}\Bigl[ \langle y_{k} - x^{*}, \nabla h(y_{k})\rangle + \hat{h}(y_{k}) - \hat{h}(x^{*})\Bigr]\\
        &\leq - F_{x^{*}}(y_{k}) - \varepsilon_{k} (H(y_{k}) - H(x^{*})) \\
        &\leq \phi_{x^{*}, p^{*}}(\varepsilon_{k}). 
    \end{align*}
Plugging this into the statement given by Lemma \ref{lem:fb-eg-iterate-descent} yields, for every $k \geq 0$,
    \[
        \| x_{k + 1} - x^{*}\|^{2} \leq \| x_{k} - x^{*}\|^{2} - (1 - s L_{k}) \Bigl[ \| x_{k} - y_{k}\|^{2} + \| x_{k + 1} - y_{k}\|^{2} \Bigr] + 2s \phi_{x^{*}, p^{*}}(\varepsilon_{k}).
    \]
The last summand of the right-hand side of the previous inequality is summable by assumption. This, together with part (iii) of Lemma \ref{lem:lower-upper-bound-nuk}, yield our claims. 
\end{proof}

We are now in a position to prove one of the main results of this section. 
\begin{proposition}\label{prop:descent-Ek}
    Let $(x_{k})_{k\geq 0}$ and $(y_{k})_{k\geq 0}$ be the sequences given by Algorithm \ref{alg:fb-reg-eg}. Assume that $V + \partial \hat{f}$ satisfies Assumption \ref{ass:Attouch-Czarnecki-inclusion}. Then
    \[
        \mathcal{V}_{k} = o\left(\frac{1}{k}\right) \ \mbox{as} \ k \to +\infty. 
    \]  
\end{proposition}
\begin{proof}
Let $x^{*}$ be an optimal solution to \eqref{eq:bilevel-where-inner-is-monotone-inclusion}.   We recall that in the proof for Lemma \ref{lem:ergodic-rate-ineq-disc}, under the boundedness of $(x_{k})_{k\geq 0}$, which now holds in light of Lemma \ref{lem:lim-xk-exists-fb-eg}, we had arrived at \eqref{eq:descent-Ek-2}, which reads
    \begin{equation*}
        \Delta E_{k} \leq \mathcal{V}_{k} - \frac{k (1 - s^{2}L_{k}^{2} - \mu C_{H}s)}{s^{2}} \| x_{k + 1} - y_{k}\|^{2} + \frac{C_{H}}{\mu s} k | \Delta \varepsilon_{k}|^{2} + \frac{2 \Delta (k (\Delta\varepsilon_{k}))}{s} H(y_{k + 1}) \quad \forall k \geq 1.
    \end{equation*}
    Now, according to Lemma \ref{lem:lim-xk-exists-fb-eg}, the sequence $(\mathcal{V}_{k})_{k\geq 0}$ is summable. The coefficient of $\| x_{k + 1} - y_{k}\|^{2}$ is eventually nonpositive. Since $(k |\Delta \varepsilon_{k}|^{2})_{k\geq 0}$ and $(| \Delta (k (\Delta \varepsilon_{k}))|)_{k\geq 0}$ are summable, and $(H(y_{k + 1}))_{k\geq 0}$ is bounded, we obtain that the remaining terms in the previous inequality are summable. Thus, on the one hand, we obtain that $ (\Delta E_{k})_{k \geq 1}$ is upper bounded by summable terms. On the other hand, since $\frac{2 k |\Delta \varepsilon_{k}|}{s} | H(y_{k})| \to 0$ as $k\to +\infty$, the sequence $(E_{k})_{k\geq 1}$ is lower bounded. This gives the existence of $\lim_{k\to +\infty} E_{k}$, and by extension of
    \[
        \lim_{k\to +\infty} (k - 1) \mathcal{V}_{k}. 
    \] 
    Since we also have $\sum_{k = 0}^{+\infty} \mathcal{V}_{k} < +\infty$, it must be the case that $\lim_{k\to +\infty} (k - 1)\mathcal{V}_{k} = 0$, which is equivalent to 
    \[
        \mathcal{V}_{k} = o \left( \frac{1}{k}\right)
    \]
    as $k\to +\infty$.
\end{proof}
Now we come to the actual rate for the full operator $\Phi_{k} + \partial g_{k}$. 
\begin{proposition}\label{prop:rate-full-operator}
Let $(x_{k})_{k\geq 0}$ and $(y_{k})_{k\geq 0}$ be the sequences given by Algorithm \ref{alg:fb-reg-eg}. Let $(\xi_{k})_{k \geq 0}$ be given by \eqref{eq:alg-2-rewritten}. Assume $V + \partial \hat{f}$ satisfies Assumption \ref{ass:Attouch-Czarnecki-inclusion}. Then
    \[
       \dist(0, \Phi_{k}(x_{k + 1})  + \partial g_{k}(x_{k + 1})) \leq  \| \Phi_{k}(x_{k + 1}) + \xi_{k + 1}\| = o \left( \frac{1}{\sqrt{k}}\right) \ \mbox{as} \ k\to +\infty.
    \]
\end{proposition}
\begin{proof}
We will show that for every sufficiently large $k \geq 0$, we have 
    \[
        \| \Phi_{k}(x_{k + 1}) + \xi_{k + 1}\|^{2} \leq \mathcal{V}_{k},
    \]
and thus the desired statement will automatically follow. Indeed, according to \eqref{eq:alg-2-rewritten}, and using that $L_{k}^{2} \leq \frac{1}{s^{2}}$ and the monotonicity of $\Phi_{k}$,  for every sufficiently large $k \geq 0$, we have 
    \begin{align*}
        \| \Phi_{k}(x_{k + 1}) + \xi_{k + 1}\|^{2} &= \Bigl\| \frac{1}{s} (x_{k} - x_{k + 1}) - \Phi_{k}(y_{k}) + \Phi_{k}(x_{k + 1})\Bigr\|^{2} \\
        &= \| \Phi_{k}(x_{k + 1}) - \Phi_{k}(y_{k})\|^{2} + \frac{1}{s^{2}} \| x_{k} - x_{k + 1}\|^{2} + \frac{2}{s} \langle \Phi_{k}(x_{k + 1}) - \Phi_{k}(y_{k}), x_{k} - x_{k + 1}\rangle \\
        &\leq L_{k}^{2} \| x_{k + 1} - y_{k}\|^{2} + \frac{1}{s^{2}} \| x_{k} - x_{k + 1}\|^{2} + \frac{2}{s} \langle \Phi_{k}(x_{k + 1}) - \Phi_{k}(y_{k}), x_{k} - x_{k + 1}\rangle \\
        &\leq \frac{1}{s^{2}} \| x_{k + 1} - y_{k}\|^{2} + \frac{1}{s^{2}} \| x_{k} - x_{k + 1}\|^{2} + \frac{2}{s} \langle \Phi_{k}(x_{k}) - \Phi_{k}(y_{k}), x_{k} - x_{k + 1}\rangle \\
        &= \mathcal{V}_{k}.
    \end{align*}
\end{proof}
We are now in a position to state and prove the main theorem of this section.
\begin{theorem}\label{thm:discrete-time-full-rates}
    Let $(x_{k})_{k\geq 0}$ be given by Algorithm \ref{alg:fb-reg-eg}, and let $x^{*}$ be an optimal solution to \eqref{eq:bilevel-where-inner-is-monotone-inclusion}. Let $(\xi_{\hat{f}, k})_{k\geq 1}$ be given by \eqref{eq:alg-2-rewritten-2}. Assume that $V + \partial \hat{f}$ satisfies Assumption \ref{ass:Attouch-Czarnecki-inclusion}. Then, the following statements hold:
    \begin{enumerate}[\rm (i)]
        \item If $\delta > \frac{1}{2}$, then $\dist(0, V(x_{k}) + \partial {\hat f}(x_k)) \leq \| V(x_{k}) + \xi_{\hat{f}, k}\| = o \left(\frac{1}{\sqrt{k}}\right)$ as $k \to +\infty$.
        \item If $0 < \delta \leq \frac{1}{2}$, then
        \begin{gather*}
             \dist(0, V(x_{k}) + \partial {\hat f}(x_k)) \leq \| V(x_{k}) + \xi_{\hat{f}, k}\| = \mathcal{O}\left(\frac{1}{k^{\delta}}\right), \\ 
             \langle x_{k} - x^{*}, V(x_{k})\rangle + \hat{f}(x_{k}) - \hat{f}(x^{*}) = o \left(\frac{1}{\sqrt{k}}\right) \quad \mbox{and} \quad
            | H(x_{k}) - H(x^{*})| = o \left( \frac{1}{k^{\frac{1}{2} - \delta}}\right)
        \end{gather*}
        as $k \to +\infty$. Furthermore, $(x_{k})_{k \geq 0}$ converges weakly to an optimal solution to the bilevel problem \eqref{eq:bilevel-where-inner-is-monotone-inclusion} as $k \to +\infty$.
    \end{enumerate}
\end{theorem}
\begin{proof}
    According to Proposition \ref{prop:rate-full-operator}, we have
    \[
        \Bigl\| (V(x_{k}) + \xi_{\hat{f}, k}) + \varepsilon_{k - 1} (\nabla h(x_{k}) + \xi_{\hat{h}, k})\Bigr\| = \| \Phi_{k - 1}(x_{k}) + \xi_{k}\| = o\left( \frac{1}{\sqrt{k}}\right) 
    \]
    as $k\to +\infty$. Due to our standing assumption, the sequence $(\nabla h(x_{k}) + \xi_{\hat{h}, k})_{k\geq 1}$ is bounded, which gives $\| V(x_{k}) + \xi_{\hat{f}, k}\| = o\left(\frac{1}{\sqrt{k}}\right)$ as $k\to +\infty$ when $\delta > \frac{1}{2}$, thus (i) holds. When $\delta \leq \frac{1}{2}$, i.e., in case (ii), we get instead $\| V(x_{k}) + \xi_{\hat{f}, k}\| = \mathcal{O}\left(\frac{1}{k^{\delta}}\right)$ as $k\to +\infty$. Moreover, by using the subgradient inequality on $\hat{f}$ and $\hat{h}$ and the gradient inequality on $h$, for every $k\geq 1$, we obtain
    \begin{align*}
        &\rph{=} 2F_{x^{*}}(x_{k}) + 2\varepsilon_{k - 1} (H(x_{k}) - H(x^{*})) \\
        &= 2\langle x_{k} - x^{*}, V(x_{k})\rangle + 2(\hat{f}(x_{k}) - \hat{f}(x_{k})) + 2\varepsilon_{k - 1} \Bigl[ h(x_{k}) - h(x^{*}) + \hat{h}(x_{k}) - \hat{h}(x^{*})\Bigr] \\
        &\leq 2\Bigl\langle x_{k} - x^{*}, (V(x_{k}) + \varepsilon_{k - 1} \nabla h(x_{k})) + (\xi_{\hat{f}, k} + \varepsilon_{k - 1} \xi_{\hat{h}, k})\Bigr\rangle \\
        &= 2\langle x_{k} - x^{*}, \Phi_{k - 1}(x_{k}) + \xi_{k}\rangle \\
        &\leq 2\| x_{k} - x^{*}\| \| \Phi_{k - 1}(x_{k}) + \xi_{k}\| = o\left( \frac{1}{k}\right)
    \end{align*}
    as $k\to +\infty$ due to the boundedness of $(x_{k})_{k\geq 0}$. Moroever, according to Lemma \ref{lem:assumptions-inclusion}, we also have 
    \begin{equation*}
        - F_{x^{*}}(x_{k}) - \varepsilon_{k - 1} (H(x_{k}) - H(x^{*})) \leq \phi_{x^{*}, p^{*}}(\varepsilon_{k - 1}) = o\left( \frac{1}{k}\right)
    \end{equation*}
    as $k\to +\infty$. Arguing just as we did in the proof of Theorem \ref{thm:continuous-time-full-rates}, we obtain 
    \[
        F_{x^{*}}(x_{k}) = o\left( \frac{1}{\sqrt{k}}\right) \quad \text{and} \quad | H(x_{k}) - H(x^{*})| = o\left(\frac{1}{k^{\frac{1}{2} - \delta}}\right)
    \]
    as $k\to +\infty$. Again, similarly as in the continuous-time case, we use the maximal monotonicity of $V + \partial \hat{f}$ to ensure that every weak sequential cluster point of $(x_{k})_{k\geq 0}$ for problem \eqref{eq:bilevel-where-inner-is-monotone-inclusion}, and then the weak lower semicontinuity of $H$ to show that is also optimal. 
\end{proof}

\subsection{Applying the algorithm to the bilevel optimization problem with a structured convex minimization problem at the lower level}\label{subsec:fenchel-bilevel}

In the introduction, we rewrote the structured convex minimization problem \eqref{eq:inner-level-Fenchel-dual}, provided a suitable regularity condition holds, in the form of a monotone inclusion \eqref{eq:bilevel-where-inner-is-monotone-inclusion}, see \eqref{eq:Fenchel-bilevel-equivalent}. If we apply Algorithm \ref{alg:fb-reg-eg} to the bilevel optimization problem having this monotone inclusion at the lower level, we obtain 
\begin{equation}\label{alg:algorithm-applied-to-fenchel-bilevel}
    (\forall k \geq 0) \quad \left\{
        \begin{aligned}
            (\bar{x}_{k}, \bar{\lambda}_{k}) &:= \prox_{s (\hat{\bm{f}} + \varepsilon_{k}\hat{\bm{h}})} \Bigl[ (x_{k}, \lambda_{k}) - s\Bigl( \bm{V}(x_{k}, \lambda_{k}) + \varepsilon_{k} \nabla \bm{h}(x_{k}, \lambda_{k})\Bigr)\Bigr], \\
            (x_{k + 1}, \lambda_{k + 1}) &:= \prox_{s (\hat{\bm{f}} + \varepsilon_{k}\hat{\bm{h}})} \Bigl[ (x_{k}, \lambda_{k}) - s\Bigl( \bm{V}(\bar{x}_{k}, \bar{\lambda}_{k}) + \varepsilon_{k} \nabla \bm{h}(\bar{x}_{k}, \bar{\lambda}_{k})\Bigr)\Bigr].
        \end{aligned}
    \right. 
\end{equation}
Since $(\hat{\bm{f}} + \varepsilon_{k}\hat{\bm{h}}) (x, \lambda) = f(x) + g^{*}(\lambda) + \varepsilon_{k} \hat{h}(x)$, the proximal step can be split, for every $(x,\lambda) \in \cH \times \cG$ and every $k \geq 0$, as 
\[
    \prox_{s (\hat{\bm{f}} + \varepsilon_{k} + \hat{\bm{h}})} (x, \lambda) = \Bigl( \prox_{s (f + \varepsilon_{k}\hat{h})} (x), \prox_{s g^{*}}(\lambda) \Bigr). 
\]
It follows that \eqref{alg:algorithm-applied-to-fenchel-bilevel} reads
\begin{equation}\label{alg:algorithm-applied-to-fenchel-bilevel-expanded}
   (\forall k \geq 0) \quad \left\{
        \begin{aligned}
            \bar{x}_{k} &:= \prox_{s (f + \varepsilon_{k} \hat{h})} \Bigl( x_{k} - s \Bigl( \nabla m(x_{k}) + A^{*}\lambda_{k} + \varepsilon_{k} \nabla h(x_{k})\Bigr)\Bigr), \\
            \bar{\lambda}_{k} &:= \prox_{s g^{*}} (\lambda_{k} + sAx_{k}), \\
            x_{k + 1} &:= \prox_{s (f + \varepsilon_{k} \hat{h})} \Bigl( x_{k} - s \Bigl( \nabla m(\bar{x}_{k}) + A^{*}\bar{\lambda}_{k} + \varepsilon_{k} \nabla h(\bar{x}_{k})\Bigr)\Bigr), \\
            \lambda_{k + 1} &:= \prox_{s g^{*}} (\lambda_{k} + s A\bar{x}_{k}).
        \end{aligned}
    \right.
\end{equation}
Provided the inner object $ \bm{F}(x, \lambda) := \langle (x, \lambda) - (x^{*}, \lambda^{*}), \bm{V}(x, \lambda)\rangle + \hat{\bm{f}}(x, \lambda) - \hat{\bm{f}}(x^{*}, \lambda^{*})$ satisfies Assumption \ref{ass:Attouch-Czarnecki-inclusion}, we can derive rates for the appropriate optimality measures. First of all, for every $k \geq 1$, write 
\[
    \bm{\xi}_{\hat{\bm{f}}, k} := (\xi_{f, k}, \xi_{g^{*}, k}) \in \partial \hat{\bm{f}}(x_{k}, \lambda_{k}) = \Bigl( \partial f(x_{k}), \partial g^{*}(\lambda_{k})\Bigr),
\]
where
\begin{equation}\label{eq:example-subdiff-primal}
    \frac{x_k - x_{k+1} - s( \nabla m(\bar{x}_{k}) + A^{*}\bar{\lambda}_{k} + \varepsilon_{k} \nabla h(\bar{x}_{k}))}{s} := \xi_{f, k + 1} + \varepsilon_{k}\xi_{\hat h, k + 1} \in \partial f(x_{k + 1}) + \varepsilon_{k} \partial \hat h(x_{k + 1})
    \end{equation}
and 
\begin{equation}\label{eq:example-subdiff-dual}
    \frac{\lambda_k - \lambda_{k+1} + s A\bar{x}_{k}}{s} := \xi_{g^{*}, k + 1} \in \partial g^{*}(\lambda_{k + 1}).
\end{equation}
Let $k \geq 1$. We have 
\begin{align}
    \Bigl\| \nabla m(x_{k}) + A^{*} \lambda_{k} + \xi_{f, k}\Bigr\| + \| \xi_{g^{*}, k} - Ax_{k}\| &\leq \sqrt{2} \: \Bigl\| \Bigl( \nabla m(x_{k}) + A^{*} \lambda_{k} + \xi_{f, k}, \ \xi_{g^{*}, k} - Ax_{k}\Bigr)\Bigr\|_{\cH\times\cG} \nonumber\\
    &= \sqrt{2}\| \bm{V}(x_{k}, \lambda_{k}) + \bm{\xi}_{\hat{\bm{f}}, k}\|_{\cH\times\cG}. \label{eq:example-rate-1}
\end{align}
Moreover, since $\xi_{f, k} \in \partial f(x_{k})$ and $\xi_{g^{*}, k} \in \partial g^{*}(\lambda_{k})$ (equivalently $\lambda_{k} \in \partial g(\xi_{g^{*}, k})$), we have 
\begin{align}
    &\rph{\leq}f(x_{k}) + g(\xi_{g^{*}, k}) + m(x_{k}) - \Bigl[ f(x^{*}) + g(Ax^{*}) + m(x^{*})\Bigr] \nonumber\\
    &\leq \langle x_{k} - x^{*}, \xi_{f, k}\rangle + \langle \xi_{g^{*}, k} - Ax^{*}, \lambda_{k}\rangle + \langle x_{k} - x^{*}, \nabla m(x_{k})\rangle \nonumber \\
    &= \Bigl\langle x_{k} - x^{*}, \ \xi_{f, k} + A^{*}\lambda_{k} + \nabla m(x_{k})\Bigr\rangle + \langle \xi_{g^{*}, k} - Ax_{k}, \lambda_{k}\rangle. \label{eq:example-rate-2}
\end{align}
We can also provide a lower bound. Since $-A^{*}\lambda^{*} \in \partial (f + m) (x^{*})$ and $\lambda^{*} \in \partial g(Ax^{*})$, 
\begin{align}
    &\rph{\geq}f(x_{k}) + g(\xi_{g^{*}, k}) + m(x_{k}) - \Bigl[ f(x^{*}) + g(Ax^{*}) + m(x^{*})\Bigr] \nonumber\\
    &\geq \langle x_{k} - x^{*}, - A^{*}\lambda^{*}\rangle + \langle \xi_{g^{*}, k} - Ax^{*}, \lambda^{*}\rangle \nonumber \\
    &= \langle \xi_{g^{*}, k} - Ax_{k}, \lambda^{*}\rangle. \label{eq:example-rate-3}
\end{align}
As a corollary of Theorem \ref{thm:discrete-time-full-rates}, we have the following result. 
\begin{theorem}
    Let $(x_{k}, \lambda_{k})_{k\geq 0}$ be given by \eqref{alg:algorithm-applied-to-fenchel-bilevel} (equivalently by \eqref{alg:algorithm-applied-to-fenchel-bilevel-expanded}) with stepsize $s$ chosen such that $0 < s < \frac{1}{\sqrt{2}( L_{\nabla m} + \| A\|)}$, let $(\xi_{f, k}, \xi_{g^{*}, k})_{k \geq 1}$ be given by \eqref{eq:example-subdiff-primal}-\eqref{eq:example-subdiff-dual}, and let $(x^*, \lambda^*)$ be an optimal solution to \eqref{eq:Fenchel-bilevel-equivalent}. Assume that $\bm{V} + \partial \hat{\bm{f}}$ satisfies Assumption \ref{ass:Attouch-Czarnecki-inclusion}. Then, the following statements hold: 
    \begin{enumerate}[\rm (i)]
        \item If $\delta > \frac{1}{2}$, then 
        \begin{align*}
            \dist(0, \partial f(x_{k}) + A^{*}\lambda_{k} + \nabla m(x_{k}) \leq \| \xi_{f, k} + A^{*}\lambda_{k} + \nabla m(x_{k})\| = & \ o\left(\frac{1}{\sqrt{k}}\right),\\ \quad  \dist(0, \partial g^{*}(\lambda_{k}) -Ax_{k}) \leq  \| \xi_{g^{*}, k} - Ax_{k}\| = & \ o\left(\frac{1}{\sqrt{k}}\right), \\ 
            \Bigl| f(x_{k}) + g(\xi_{g^{*}, k}) + m(x_{k}) - (f(x^{*}) + g(Ax^{*}) + m(x^{*}))\Bigr| = & \ o\left(\frac{1}{\sqrt{k}}\right)
        \end{align*}
        as $k \to +\infty$.
        \item If $0 < \delta \leq \frac{1}{2}$, then 
        \begin{align*}
            \dist(0, \partial f(x_{k}) + A^{*}\lambda_{k} + \nabla m(x_{k}) \leq \| \xi_{f, k} + A^{*}\lambda_{k} + \nabla m(x_{k})\| = & \ \mathcal{O}\left(\frac{1}{k^{\delta}}\right),\\ 
            \dist(0, \partial g^{*}(\lambda_{k}) -Ax_{k}) \leq \| \xi_{g^{*}, k} - Ax_{k}\| = & \ \mathcal{O}\left(\frac{1}{k^{\delta}}\right), \\ 
            \Bigl| f(x_{k}) + g(\xi_{g^{*}, k}) + m(x_{k}) - (f(x^{*}) + g(Ax^{*}) + m(x^{*}))\Bigr| = & \ \mathcal{O}\left(\frac{1}{k^{\delta}}\right) \\
            \text{and} \quad | H(x_{k}) - H(x^{*})| = o\left(\frac{1}{k^{\frac{1}{2} - \delta}}\right)
        \end{align*}
        as $k \to +\infty$. Furthermore, $(x_{k}, \lambda_{k})_{k \geq 0}$ converges weakly to an optimal solution to \eqref{eq:Fenchel-bilevel-equivalent}  as $k \to +\infty$.
    \end{enumerate}
\end{theorem}

\begin{remark}
    Recall that 
    \[
        u^{*} := - A^{*}\lambda^{*} - \nabla m(x^{*}) \in \partial f(x^{*}) \quad \text{and} \quad Ax^{*} \in \partial g^{*}(\lambda^{*}).
    \]
    The inner object reads, for every $(x, \lambda) \in \cH \times \cH$,
    \begin{align*}
        \bm{F}_{(x^{*}, \lambda^{*})}(x, \lambda) &= \Bigl\langle (x, \lambda) - (x^{*}, \lambda^{*}), \bm{V}(x, \lambda)\Bigr\rangle + \hat{\bm{f}}(x, \lambda) - \hat{\bm{f}}(x^{*}, \lambda^{*}) \\
        &= \langle x - x^{*}, \nabla m(x) + A^{*}\lambda\rangle + \langle \lambda - \lambda^{*}, -Ax\rangle + f(x) - f(x^{*}) + g^{*}(\lambda) - g^{*}(\lambda^{*}) \\
        &= \Bigl[f(x) - f(x^{*}) - \langle x - x^{*}, u^{*}\rangle \Bigr] + \Bigl[ g^{*}(\lambda) - g^{*}(\lambda^{*}) - \langle \lambda - \lambda^{*}, Ax^{*}\rangle\Bigr]\\
        &\rph{=} + \Bigl\langle x - x^{*}, u^{*} + \nabla m(x) + A^{*}\lambda\Bigr\rangle + \langle \lambda - \lambda^{*}, Ax^{*} - Ax\rangle \\
        &= \Bigl[f(x) - f(x^{*}) - \langle x - x^{*}, u^{*}\rangle \Bigr] + \Bigl[ g^{*}(\lambda) - g^{*}(\lambda^{*}) - \langle \lambda - \lambda^{*}, Ax^{*}\rangle\Bigr] \\
        &\rph{=} \langle x - x^{*}, \nabla m(x) - \nabla m(x^{*})\rangle \\
        &= D_{f}(x, x^{*}) + D_{g^{*}}(\lambda, \lambda^{*}) + \langle x - x^{*}, \nabla m(x) - \nabla m(x^{*})\rangle \\
        &\geq D_{f}(x, x^{*}) + D_{g^{*}}(\lambda, \lambda^{*})
    \end{align*}
    if we are to interpret the $f$ and $g^{*}$ terms as Bregman divergences of sorts. This suggests that if we wish to check whether the sharpness condition \eqref{eq:sharpness-discrete} holds for $\bm{V} + \partial \hat{f}$, we could have $D_{f}(x, x^{*}) + D_{g^{*}}(\lambda, \lambda^{*})$ on the right-hand side as opposed to $\bm{F}_{(x^{*}, \lambda^{*})}(x, \lambda)$.  
\end{remark}

\section{Numerical experiments}

In this section, we test Algorithm \ref{alg:fb-reg-eg} on two benchmark bilevel programs. Experiments are performed in Python on a 12thGen. Intel(R) Core(TM) i7–1255U, 1.70–4.70 GHz laptop with 16 Gb of RAM and are available for reproducibility at \href{https://github.com/echnen/bilevel-on-monotone}{https://github.com/echnen/bilevel-on-monotone}.
	
As a numerical baseline, we consider the Optimistic Extragradient Method for Hierarchical Optimization (OEG-H) introduced in \cite{DvurechenskyMarschnerShternStaudigl2026}. To the best of our knowledge, this is currently the only available method capable of addressing problems with the generality of \eqref{eq:bilevel-inclusion-intro}. The method can be viewed as the counterpart of Algorithm~\ref{alg:fb-reg-eg}: While Algorithm~\ref{alg:fb-reg-eg} extends Korpelevich's extragradient method \cite{Korpelevich1976} to the hierarchical setting, OEG-H is of Popov type \cite{Popov1980} and takes the form
\begin{equation}
	(\forall k \geq 0) \quad \left\{
		\begin{aligned}
			&y_k = \prox_{s g_k}\bigl(x_k - s \Phi_k(y_{k-1})\bigr)\,,\\
			&x_{k+1} = \prox_{s g_k}\bigl(x_k - s \Phi_k(y_k)\bigr)\,,
		\end{aligned}
	\right.
\end{equation}
and starting points $x_0, y_{-1}\in \cH$. In particular, OEG-H requires only a single evaluation of the Lipschitz continuous monotone operator per iteration, but it requires storing one auxiliary variable and is subject to the more restrictive step-size condition ($s \leq \frac{1}{4L_V}$)---the usual trade-off between Korpelevich- and Popov-type extragradient schemes.

\subsection{Comparison on an illustrative example}\label{sec:illustrative_example}

We consider the bilevel optimization problem \eqref{eq:bilevel-inclusion-intro} defined in $\cH=\R^d$, with $d\in \N$, as
\begin{equation}\label{eq:bilevel_toy_example}
	\begin{aligned}
		\min_{x \in \R^d} & \qquad \mathrm{tv}_1(x):= \sum_{i=1}^{d-1} |x_{i+1}-x_i|\,,\\
		\text{s.t.:} &  \qquad 0 = V(x) := Ax-b+x-P_C(x)\,,
	\end{aligned}
\end{equation}
where $P_C$ is the orthogonal projection onto the closed convex set $$C:=\left\{ x\in\mathbb{R}^d \mid x_i\leq \frac12 \quad\text{for all} \ i=m+1,\ldots,d \right\}\,,$$ for $m < d$ fixed, and $A$ is the skew-symmetric matrix defined by
\begin{equation}
A_{i+1,i}=1\,, \qquad A_{i,i+1}=-1\,,\qquad i\in \{1,\ldots,m-1\}\,,
\end{equation}
with all remaining entries equal to zero, and $b\in \R^d$ is zero except in $b_2=20$ and $b_{m-1} = -1$.  The nonzero tridiagonal block in $A$ has dimension $m$. The set of solutions to the inner problem is thus
\begin{equation}
\mathcal{S} := \{x \in \R^d \mid x_1 = 20\,, \ x_m = 1\,, x_i = 0 \ \text{for} \ i =2, \dots, m-1\,, \ x_i \leq 1/2 \ \text{for} \ i > m\}\,.
\end{equation}
The outer function, thus selects $x \in \mathcal{S}$ with $x_i = 1/2$ for all $i >m$, which is the explicit solution to the bilevel program \eqref{eq:bilevel_toy_example}. Note that \eqref{eq:bilevel_toy_example} fits the general framework \eqref{eq:bilevel-inclusion-intro} with $\hat h := \mathrm{tv}_1$, $h=0$ and $\hat f=0$. The proximity operator of $\mathrm{tv}_1$ can be computed in almost linear time \cite{Condat2013}, and $V$ is a nonlinear monotone and Lipschitz continuous operator.

\subsubsection{Comparison with baseline} To solve \eqref{eq:bilevel_toy_example}, we compare Bi-EG against OEG-H. For both the method, we set $\varepsilon_{k} := \frac{c}{(k + a)^{\delta}}$, for every $k \geq 0$, with $c, a > 0$ and $k\geq 0$, ranging across three choices of $\delta$ as $\delta=0.2, 0.35, 0.5$. We set near maximum step-size allowed setting $s = 0.99 / L$ and $s = 0.249 / L$ for Bi-EG and OEG-H, respectively. Additionally, we consider 20 different random initialization $x_0$, and set $y_{-1}= x_0$ for OEG-H. We run the two methods for $10^4$ iterations measuring i) inner gap, ii) inner residual and iii) outer objective gap, i.e., the three quantities in Point (iii) of Theorem \ref{thm:discrete-time-full-rates}, and show the results in Figure \ref{fig:illustrative_example}.

\begin{figure}[t]
	\centering
	\includegraphics[width=0.9\linewidth]{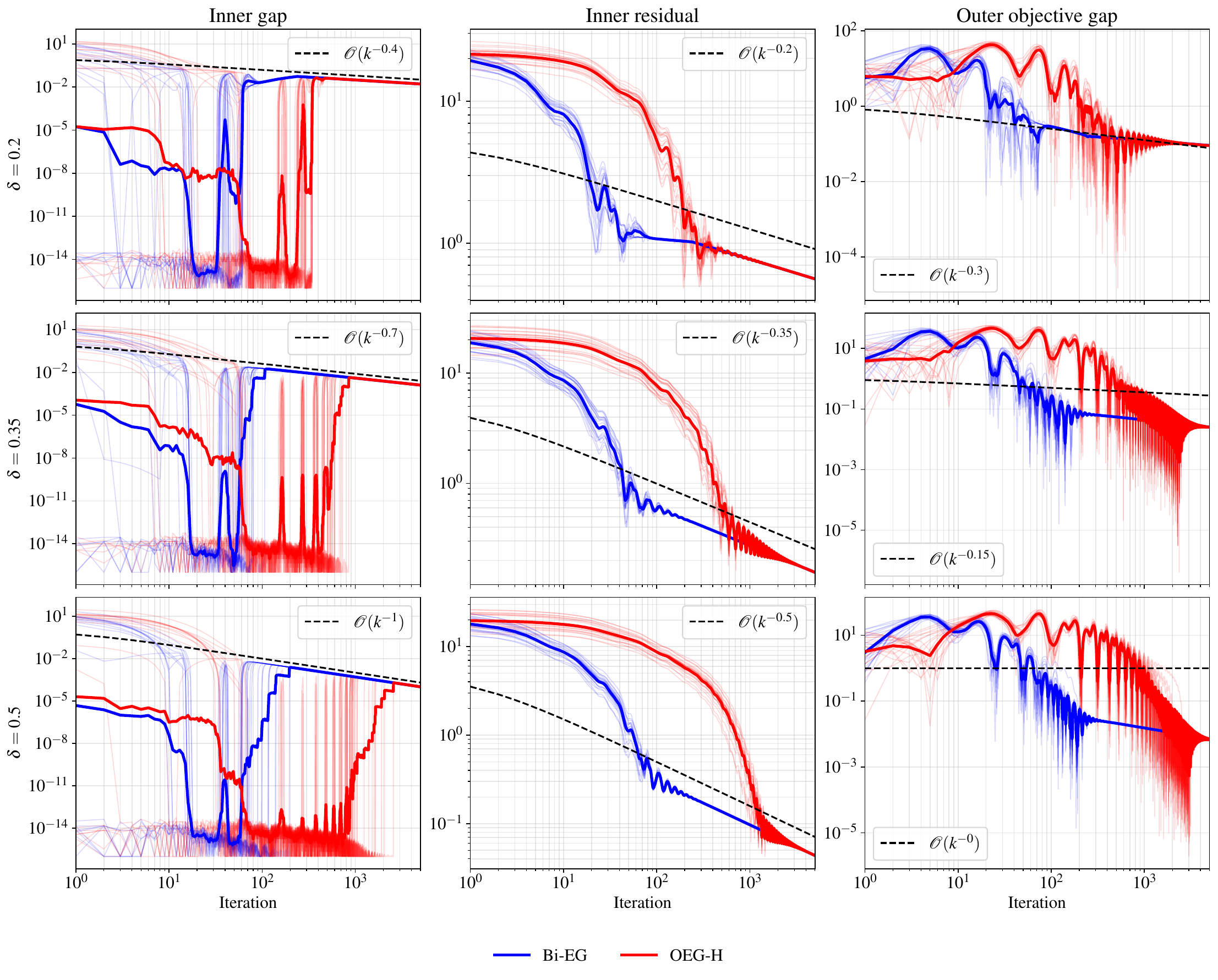}
	\caption{Comparison between Bi-EG and OEG-H to solve \eqref{eq:bilevel_toy_example}.}
	\label{fig:illustrative_example}
\end{figure}

Figure~\ref{fig:illustrative_example} shows that Bi-EG reaches high accuracy approximately one order of magnitude faster than its Popov-type counterpart, OEG-H, while the two methods exhibit comparable performance over longer time horizons. This faster initial convergence comes at the cost of one additional operator evaluation per iteration. On the other hand, Bi-EG requires storing only one iterative variable, whereas OEG-H stores two. Overall, the two methods appear numerically comparable once both computational cost and memory requirements are taken into account.

Additionally, Figure~\ref{fig:illustrative_example} shows that the theoretical convergence rates established for Bi-EG accurately capture its observed numerical behavior across the different choices of $\delta$. Interestingly, OEG-H exhibits qualitatively similar behavior in these experiments. However, to the best of our knowledge, no existing theoretical result establishes analogous convergence rates for OEG-H under the general geometric setting considered in Theorem \ref{thm:ergodic-rate-discrete}.

\subsubsection{Testing dependence on \texorpdfstring{$\delta$}{delta}}\label{sec:num_testing_dependence_on_delta}

We next investigate more closely the influence of the decay exponent $\delta$ in the regularization sequence $(\varepsilon_k)_{k \geq 0}$. We consider $20$ values of $\delta$ ranging uniformly in $(0.1, 0.5)$ and implement Bi-EG using the same initial point in each run. The parameter \(\delta\) governs the trade-off between the decay of the lower-level residual and the convergence of the upper-level objective, according to Theorem \ref{thm:ergodic-rate-discrete}.

This behavior is visible in Figure~\ref{fig:illustrative_example}. Indeed smaller values of $\delta$ favor inner convergence, as it is clear from the inner gap and inner residual decay. On the other hand, from the outer objective gap decay, we can also deduce that higher values of $\delta$ favor outer convergence, as residuals corresponding to larger exponents are actually larger than those with smaller ones. 

\begin{figure}[t]
	\includegraphics[width=\linewidth]{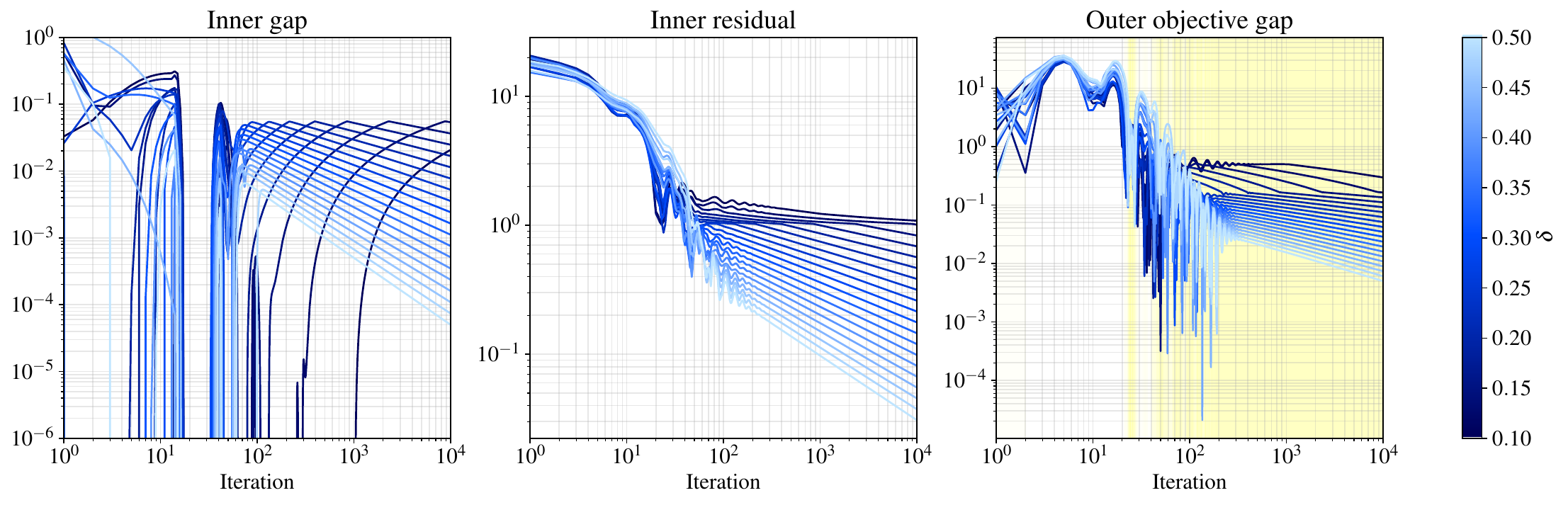}
	\caption{Dependence on $\delta$. In the outer-objective plot, the yellow regions indicate iterations for which the objective gap is negative; accordingly, the plotted $y$-values in these regions should be interpreted with the opposite sign.}
\end{figure}

\section{A secondary variational problem in optimal transport}

We now consider a problem closely related to the so-called \emph{secondary variational problem} from Optimal Transport (OT) that naturally lies in the bilevel optimization setting \cite[Chapter 3.1.2]{Santambrogio2015}.

We consider two uniform discrete probability measures supported on two regular $N$-gons:
\[
\mu:=\frac{1}{N}\sum_{i=1}^{N}\delta_{x_i},
\qquad
\nu:=\frac{1}{N}\sum_{j=1}^{N}\delta_{y_j},
\]
where $x_i,y_j\in\R^2$. More precisely, $x_i$ and $y_j$ lie on the unit circle at angles $2\pi i/N$ and $2\pi j/N+\pi/N$, respectively. Thus, the support of $\nu$ is obtained by rotating the support of $\mu$ by an angle $\pi/N$, see also Figure \ref{fig:ot-plans} for an illustration. 

In this example, both the optimal transport plans and the corresponding optimal potentials can be characterized in closed form. The optimal plans split the mass between adjacent target nodes in a cyclic fashion: Each source node sends an amount $m\in[0, 1/n]$ to the left adjacent node and the remaining mass $1/N-m$ to the right adjacent node. Among all optimal couplings between $\mu$ and $\nu$, we aim to select the one that minimizes the mass transported along a fixed connection, leading to the following bilevel problem:
\begin{equation}\label{eq:bilevel_ot}
	\left\{\begin{aligned}
	&\min_{\gamma \in \R^{n\times n}}  \quad \tfrac{1}{2}\gamma_{11}^2\\
	& \quad \text{s.t.:} \quad \gamma \in \text{OT}(\mu, \nu)\,,
	\end{aligned}\right. 
	\qquad \text{where} \qquad  
	\left\{
	\begin{aligned}
	&\text{OT}(\mu, \nu) := \argmin_{\gamma \in \R^{n\times n}} \qquad \sum_{i, j=1}^n c_{ij} \gamma_{ij}\\
	& \quad\text{s.t.:}  \quad \sum_{i=1}^n\gamma_{ij} = \nu_j\,, \quad \text{for all} \ j \in \{1, \dots, n\}\,,\\
	& \phantom{\quad \text{s.t.:}}\quad \sum_{j=1}^n\gamma_{ij} = \mu_i\,, \quad \text{for all} \ i \in \{1, \dots, n\}\,,\\
	&  \phantom{\quad\text{s.t.:}} \quad \gamma_{ij} \geq 0\,, \quad \text{for all} \ i, j \in \{1, \dots, n\}\,,
\end{aligned}\right. 
\end{equation} 
where $c_{ij}:= \|x_i - y_j\|^2$ for all $i, j$. By construction, minimizing $\gamma_{11}$ over $\operatorname{OT}(\mu,\nu)$ yields a unique bilevel-optimal transport plan, which is induced by a Monge map. The corresponding dual potential, however, need not be unique. To avoid overburdening the presentation, we omit the straightforward derivation of the optimal plan and the associated family of dual potentials.

We reformulate this problem as an instance of \eqref{eq:bilevel-inclusion-intro} by expressing the lower-level problem through its primal-dual optimality conditions. Let $A\in\R^{2n\times n^2}$ denote the matrix encoding the marginal constraints, and let $b\in\R^{2n}$ be the corresponding right-hand side. After vectorizing the transport plan, the lower-level problem reads
\begin{equation}
	\min_{\gamma\in\R^{n^2}}
	\left\{
	\langle c,\gamma\rangle
	+\iota_{\R^{n^2}_+}(\gamma)
	+\iota_{\{b\}}(A\gamma)
	\right\},
	\label{eq:ot-lower-level}
\end{equation}
where $c\in\R^{n^2}$ is the vectorized transportation-cost matrix. Problem \eqref{eq:ot-lower-level} fits the formulation in \eqref{eq:inner-level-Fenchel-dual} and can therefore be written as a monotone inclusion of the form \eqref{eq:bilevel-inclusion-intro}.

We apply both Bi-EG and OEG-H to the resulting bilevel inclusion. Figure~\ref{fig:ot-plans} compares the transport plan obtained by solving only the lower-level problem (with $H=0$) with the bilevel-optimal plans reconstructed by Bi-EG and OEG-H. The convergence behavior of the two bilevel methods is reported in Figure~\ref{fig:ot-comparison}.

The faster convergence of Bi-EG observables from the residual decay in Figure~\ref{fig:ot-comparison}, results in a more accurate reconstruction of the bilevel-optimal coupling in Figure~\ref{fig:ot-plans}. In particular, the Bi-EG solution is closer to the expected Monge structure of the selected optimal plan, whereas the OEG-H solution still displays residual mass splitting between adjacent target nodes.

\begin{figure}[t]
	\centering
	\includegraphics[width=0.9\linewidth]{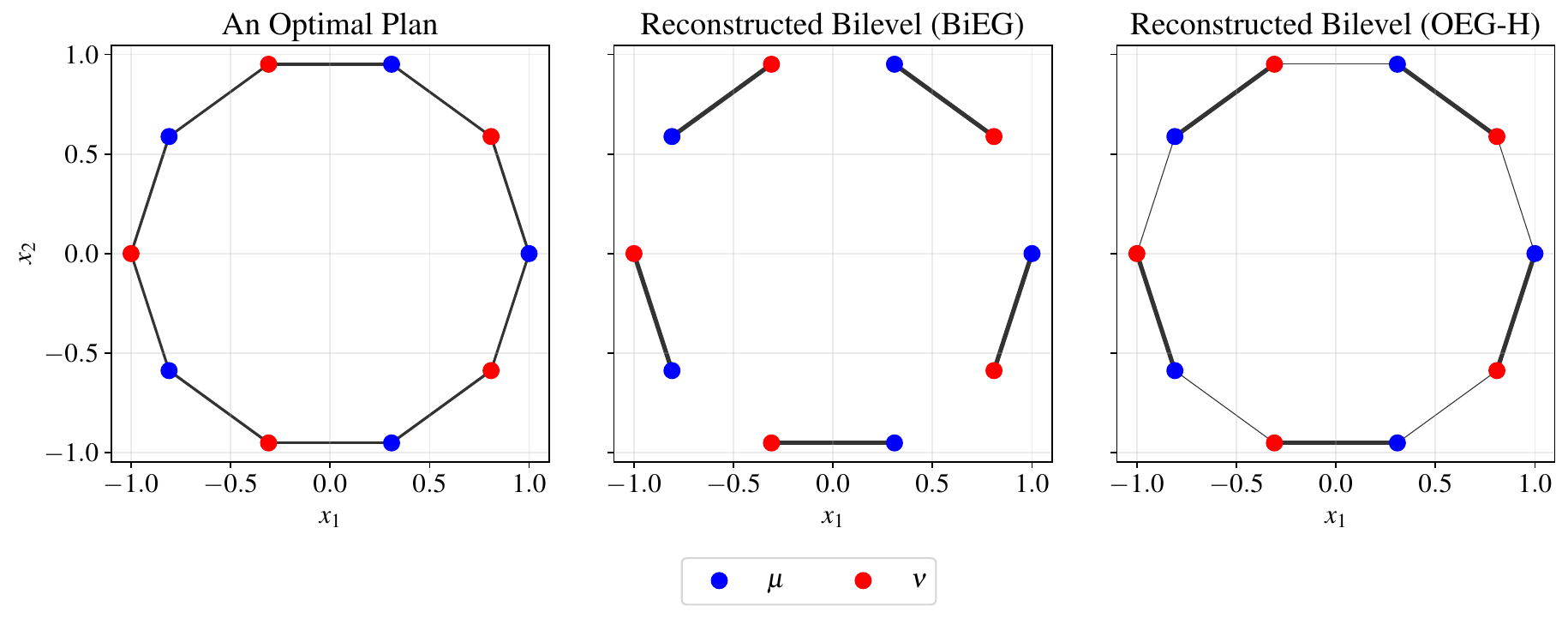}
	\caption{Reconstructed transport plans. From left to right: a solution to the lower-level optimal-transport problem, the bilevel solution reconstructed by Bi-EG, and the bilevel solution reconstructed by OEG-H. While the lower-level problem admits multiple optimal couplings, the upper-level objective selects the unique Monge plan that minimizes the mass transported along the prescribed connection.}
	\label{fig:ot-plans}
\end{figure}

\begin{figure}[t]
	\centering
	\includegraphics[width=.9\linewidth]{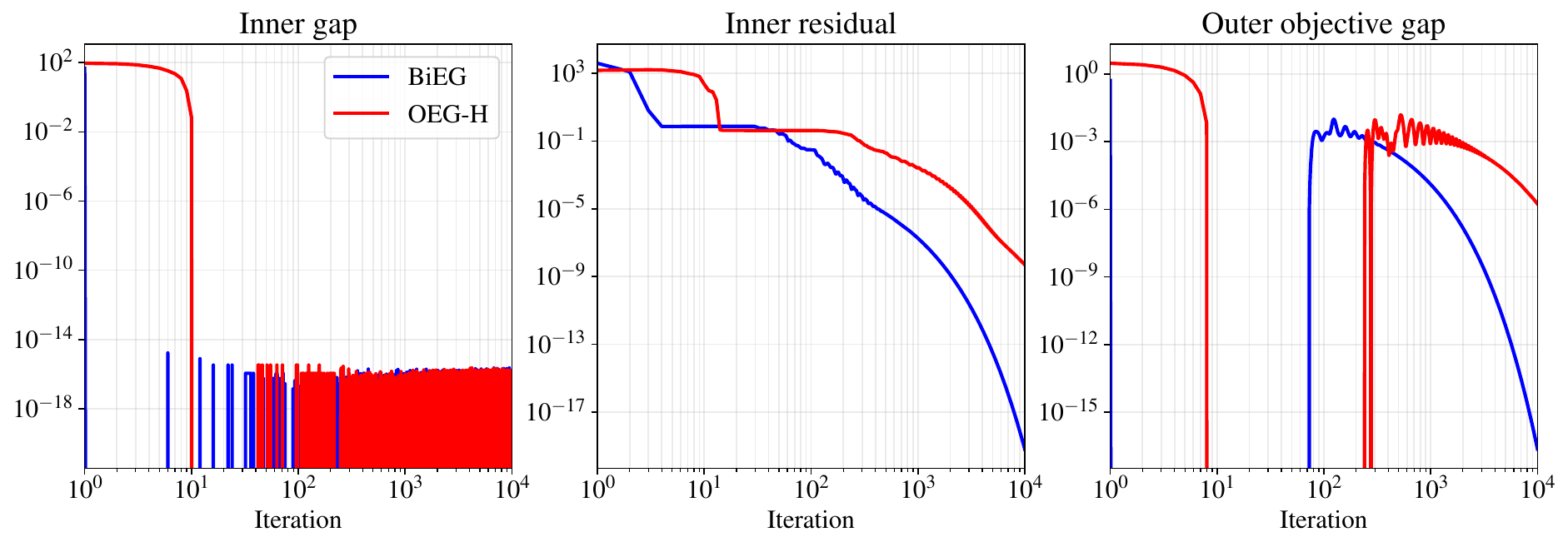}
	\caption{Convergence behavior of Bi-EG and OEG-H for the bilevel optimal-transport problem.}
	\label{fig:ot-comparison}
\end{figure}

\section*{AI usage statement}
Some of the proofs in this work where assisted by the GPT-5.6 and GPT-6 family of LLMs from OpenAI. All results were verified by the authors, who take full responsibility for the contents of the manuscript.

\printbibliography

% --------------------------------------------------------------
%     You don't have to mess with anything below this line.
% --------------------------------------------------------------

\end{document}